\RequirePackage[l2tabu, orthodox]{nag} 

\documentclass[reqno]{amsart}

\pdfoutput=1
\usepackage[utf8]{inputenc} 
\usepackage{commath} 
\usepackage{cases} 
\usepackage{mathtools} 
\usepackage[colorlinks, citecolor=red, breaklinks=true]{hyperref} 
\usepackage[stretch=10, shrink=10]{microtype} 
\usepackage[initials,msc-links,nobysame]{amsrefs}[2007/10/22]

\newtheorem{theorem}{Theorem}[section]
\newtheorem{lemma}[theorem]{Lemma}

\newtheorem{proposition}[theorem]{Proposition}
\newtheorem{fact}[theorem]{Fact}

\theoremstyle{definition}
\newtheorem{definition}[theorem]{Definition}

\theoremstyle{remark}
\newtheorem{remark}[theorem]{Remark}

\numberwithin{equation}{section}

\DeclareMathOperator{\Span}{span}
\DeclareMathOperator{\tr}{tr}

\DeclareMathOperator{\im}{Im}
\let\Re\relax
\DeclareMathOperator{\Re}{Re}

\usepackage{paralist,enumitem}
\setlist{listparindent=0pt,parsep=3pt}

\newcommand{\TitleWithUrl}[1]{\IfEmptyBibField{doi}%
  {\IfEmptyBibField{url}{\textit{#1}}%
    {\IfEmptyBibField{eprint}{\href {\BibField{url}}{\textit{#1}}}{\textit{#1}}}%
    }%
  {\href {https://doi.org/\BibField{doi}}{\textit{#1}}}}
\renewcommand{\eprint}[1]{\IfEmptyBibField{url}{\url{#1}}%
  {\href {\BibField{url}}{#1}}}

\BibSpec{article}{%
    +{}  {\PrintAuthors}                {author}
    +{,} { \TitleWithUrl}               {title}
    +{.} { }                            {part}
    +{:} { \textit}                     {subtitle}
    +{,} { \PrintContributions}         {contribution}
    +{.} { \PrintPartials}              {partial}
    +{,} { }                            {journal}
    +{}  { \textbf}                     {volume}
    +{}  { \PrintDatePV}                {date}
    +{,} { \issuetext}                  {number}
    +{,} { \eprintpages}                {pages}
    +{,} { }                            {status}
    +{,} { available at arXiv:\eprint}        {eprint}
    +{}  { \parenthesize}               {language}
    +{}  { \PrintTranslation}           {translation}
    +{;} { \PrintReprint}               {reprint}
    +{.} { }                            {note}
    +{.} {}                             {transition}
    +{}  {\SentenceSpace \PrintReviews} {review}
}

\BibSpec{collection.article}{%
    +{}  {\PrintAuthors}                {author}
    +{,} { \TitleWithUrl}                     {title}
    +{.} { }                            {part}
    +{:} { \textit}                     {subtitle}
    +{,} { \PrintContributions}         {contribution}
    +{,} { \PrintConference}            {conference}
    +{}  {\PrintBook}                   {book}
    +{,} { }                            {booktitle}
    +{,} { \PrintDateB}                 {date}
    +{,} { pp.~}                        {pages}
    +{,} { }                            {status}
    +{,} { available at \eprint}        {eprint}
    +{}  { \parenthesize}               {language}
    +{}  { \PrintTranslation}           {translation}
    +{;} { \PrintReprint}               {reprint}
    +{.} { }                            {note}
    +{.} {}                             {transition}
    +{}  {\SentenceSpace \PrintReviews} {review}
}
\BibSpec{book}{%
    +{}  {\PrintPrimary}                {transition}
    +{,} { \TitleWithUrl}               {title}
    +{.} { }                            {part}
    +{:} { \textit}                     {subtitle}
    +{,} { \PrintEdition}               {edition}
    +{}  { \PrintEditorsB}              {editor}
    +{,} { \PrintTranslatorsC}          {translator}
    +{,} { \PrintContributions}         {contribution}
    +{,} { }                            {series}
    +{,} { \voltext}                    {volume}
    +{,} { }                            {publisher}
    +{,} { }                            {organization}
    +{,} { }                            {address}
    +{,} { \PrintDateB}                 {date}
    +{,} { }                            {status}
    +{}  { \parenthesize}               {language}
    +{}  { \PrintTranslation}           {translation}
    +{;} { \PrintReprint}               {reprint}
    +{.} { }                            {note}
    +{.} {}                             {transition}
    +{}  {\SentenceSpace \PrintReviews} {review}
}

\BibSpec{thesis}{%
    +{}  {\PrintAuthors}                {author}
    +{.} { \TitleWithUrl}               {title}
    +{:} { \textit}                     {subtitle}
    +{,} { \PrintThesisType}            {type}
    +{,} { }                            {organization}
    +{} { \PrintDate}                  {date}
    +{,} { }                            {address}
    +{,} { \eprint}                     {eprint}
    +{,} { }                            {status}
    +{}  { \parenthesize}               {language}
    +{}  { \PrintTranslation}           {translation}
    +{;} { \PrintReprint}               {reprint}
    +{.} { }                            {note}
    +{.} {}                             {transition}
    +{}  {\SentenceSpace \PrintReviews} {review}
}

\title[Zero mean curvature surfaces in $\mathbb{I}^3$ with planar curvature lines]{Zero mean curvature surfaces in isotropic space with planar curvature lines}

\author{Joseph Cho}
\address[Joseph Cho]{Global Leadership School, Handong Global University, 558 Handong-ro Buk-gu, Pohang, Gyeongsangbuk-do 37554, Republic of Korea}
\email{jcho@handong.edu}

\author{Masaya Hara}
\address[Masaya Hara]{Department of Mathematics, Graduate School of Science, Kobe university, 1-1 Rokkodai-cho, Nada-ku, Kobe 657-8501, Japan}
\email{mhara@math.kobe-u.ac.jp}

\begin{document}

\begin{abstract}
	We give a comprehensive account of zero mean curvature surfaces with planar curvature lines in isotropic $3$-space.
	After giving a complete classification all such surfaces, we show that they belong to a $1$-parameter family of surfaces. We then investigate their relationship to Thomsen-type surfaces in isotropic $3$-space, those zero mean curvature surfaces in isotropic $3$-space that are also affine minimal.
\end{abstract}
\maketitle

\section{Introduction}
In surface theory, surfaces with specific geometric constraints—such as constant curvature, embeddedness, or compactness—are highly sought after, with particular interest in those satisfying multiple constraints.
Among these, the condition of having \emph{planar curvature lines} has drawn significant attention from both classical and modern geometers.
The study of surfaces with planar curvature lines began with the works of Monge \cite{monge_application_1850}, Joachimsthal \cite{joachimsthal_demonstrationes_1846}, and Bonnet \cite{bonnet_memoire_1848, bonnet_memoire_1867}, who treated surfaces as graphs of functions of two variables.
However, it was Enneper \cite{enneper_analytisch-geometrische_1868, enneper_untersuchungen_1878} and his students \cite{bockwoldt_ueber_1878, lenz_ueber_1879, voretzsch_untersuchung_1883} who introduced analytical methods for studying surfaces of constant curvature with planar curvature lines, and these methods had a profound impact on modern surface theory.
Notably, following the discovery of Wente tori \cite{wente_counterexample_1986}, Abresch \cite{abresch_constant_1987} and Walter \cite{walter_explicit_1987} independently observed that Wente tori possess one family of planar curvature lines, providing a simplifying ansatz that enabled explicit descriptions of these tori following the approach taken by Enneper.

Classical geometers were also interested in minimal surfaces with planar curvature lines, with their classification completed by Bonnet \cite{bonnet_observations_1855} and Enneper \cite{enneper_analytisch-geometrische_1864}.
For this reason, such minimal surfaces are often referred to as \emph{Bonnet minimal surfaces}.
(For a modern treatment, see \cite[Section~2.6]{nitsche_vorlesungen_1975}.)
Interestingly, Bonnet minimal surfaces have an alternative characterization from the perspective of affine differential geometry: a minimal surface is a Bonnet minimal surface if and only if its conjugate minimal surface is a Thomsen surface, a minimal surface that is also affine minimal \cite{thomsen_uber_1923}.
Through analytical methods, it was later demonstrated that Thomsen surfaces form a 1-parameter family of surfaces \cite{schaal_ennepersche_1973, barthel_thomsensche_1980}, which in turn implies that Bonnet minimal surfaces also constitute a 1-parameter family (see also \cite{cho_deformation_2017}).

Similar considerations were given to zero mean curvature surfaces in Minkowski space.
The classification of Bonnet-type maximal surfaces is provided in \cite{leite_surfaces_2015}, while their connection to Thomsen-type maximal surfaces is explored in \cite{manhart_bonnet-thomsen_2015}, and they are also shown to form a 1-parameter family in \cite{cho_deformation_2018}.
Thomsen-type timelike minimal surfaces were first studied in \cite{magid_timelike_1991-1}, and Bonnet-type timelike minimal surfaces received a comprehensive treatment, including their classification, in \cite{akamine_analysis_2020}.

In this paper, we shift our focus to \emph{isotropic space} to explore analogous problems.
Isotropic plane geometry first emerged as one of the Cayley-Klein geometries through a particular choice of the absolute quadric in projective geometry \cite{beck_zur_1913, berwald_uber_1915}.
The study of surface theory within isotropic geometry followed soon after, with foundational works by Strubecker \cite{strubecker_differentialgeometrie_1942, strubecker_differentialgeometrie_1942-1, strubecker_differentialgeometrie_1944, strubecker_differentialgeometrie_1949}.
Notably, a Weierstrass-type representation for zero mean curvature surfaces in isotropic space was introduced in \cite{strubecker_differentialgeometrie_1942-1}.
(For a more contemporary account of isotropic geometry, see \cite{sachs_isotrope_1990}.)
In recent years, interest in surface theory within isotropic geometry has seen renewed attention, particularly due to its applications in architecture \cite{jiang_planar_2022, kilian_material-minimizing_2017, millar_designing_2023, pottmann_laguerre_2009, pottmann_discrete_2007, tellier_designing_2023}.

Our aim of this paper is to provide an exhaustive account of zero mean curvature surfaces with planar curvature lines in isotropic space.
This problem was considered by Strubecker \cite{strubecker_uber_1977} as part of his examination of zero mean curvature surfaces in isotropic space that are also affine minimal.
In the work, Strubecker finds notable examples of such surfaces, and shows that the conjugate zero mean curvature surfaces are zero mean curvature surfaces with planar curvature lines.

We will consider this topic with an emphasis on the constraint of having planar curvature lines, allowing us to use modern analytical methods modeled after the works of Abresch and Walter \cite{abresch_constant_1987, walter_explicit_1987} (also used in \cite{akamine_analysis_2020, cho_deformation_2017, cho_deformation_2018}).
Thus, after briefly describing surface theory of isotropic $3$-space in Section~\ref{sect:two}, we will focus on zero mean curvature surfaces with planar curvature lines in Section~\ref{sect:three}.
In particular, we will use the aforementioned analytical methods to give an analytic classification of all zero mean curvature surfaces with planar curvature lines in Theorem~\ref{thm:analytic}.
Then using the Weierstrass-type representation for zero mean curvature surfaces in isotropic space \cite{strubecker_differentialgeometrie_1942-1} (see also \cite{sachs_isotrope_1990, pember_weierstrass-type_2020, da_silva_holomorphic_2021, seo_zero_2021}), we will recover the Weierstrass data for these surfaces, thereby obtaining a complete classification in Theorem~\ref{thm:wdata}.
After noting few geometric facts about zero mean curvature surfaces with planar curvature lines in Section~\ref{sect:four}, we show that these surfaces constitute a $1$-parameter family of surfaces in Section~\ref{sect:five} on the level of Weierstrass-data (see Theorem~\ref{thm:deformation}).

Finally in Section~\ref{sect:six}, we explore zero mean curvature surfaces in isotropic space that are also affine minimal, and show that they are conjugate zero mean curvature surfaces of those zero mean curvature surfaces with planar curvature lines in Theorem~\ref{thm:conjugate}, recovering the result by Strubecker in \cite{strubecker_uber_1977}.
This result quickly leads to a complete classification of these surfaces in Theorem~\ref{thm:wdata2}, and also allows us to conclude immediately that these surfaces also constitute a $1$-parameter family of surfaces in Theorem~\ref{thm:deformation2}.

\textbf{Acknowledgements.} The authors would like to express their gratitude to Shintaro Akamine and Yuta Ogata for their helpful comments. The second author gratefully acknowledges the support from JST SPRING, Grant Number JPMJSP2148, including the opportunity to visit TU Wien.

\section{Preliminaries}\label{sect:two}
We first briefly describe the isotropic $3$-space and the surface theory within by following the viewpoint taken in \cite{cho_spinor_2024}.
\subsection{Isotropic \texorpdfstring{$3$}{3}-space}
Denote by $\mathbb{R}^{3,1}$ the Minkowski $4$-space equipped with bilinear form $\langle \cdot, \cdot \rangle$ of signature $(- + + +{})$, and let $\mathcal{L}$ denote the light cone, i.e.\
    \[
        \mathcal{L} = \{X \in \mathbb{R}^{3,1} : \langle X, X \rangle = 0\}.
    \]
Choosing some $\mathfrak{p}, \tilde{\mathfrak{p}} \in \mathcal{L}$ with $\langle \mathfrak{p}, \tilde{\mathfrak{p}} \rangle = 1$, the isotropic $3$-space $\mathbb{I}^3$ is given by
    \[
        \mathbb{I}^3 = \{X \in \mathbb{R}^{3,1} : \langle X, \mathfrak{p} \rangle = 0\}.
    \]
To consider coordinates, we will normalize $\mathfrak{p}$ and $\tilde{\mathfrak{p}}$ as
    \[
        \mathfrak{p} = (1,0,0,1)^t \quad\text{and}\quad \tilde{\mathfrak{p}} = \frac{1}{2}(-1,0,0,1)^t,
    \]
and the isotropic space is given by
    \[
        \mathbb{I}^3 = \{ (\mathbf{l}, \mathbf{x}, \mathbf{y}, \mathbf{l})^t \in \mathbb{R}^{3,1} \},
    \]
so that the induced metric is
    \[
        \dif{s}^2_{\mathbb{I}^3} = \dif{\mathbf{x}}^2 + \dif{\mathbf{y}}^2.
    \]
We will often refer to the coordinate space $\{(\mathbf{l}, \mathbf{x}, \mathbf{y})^t\}$ with the above metric as isotropic $3$-space, and refer to the $\mathbf{l}$-direction as \emph{vertical}.

\subsection{Surface theory}
Now for some simply-connected domain $\Sigma$, let $X: \Sigma \to \mathbb{I}^3$ be a spacelike immersion, that is, the induced metric on the tangent planes of $X$ is Riemannian.
Thus, we may locally find conformal coordinates $(u,v) \in \Sigma$ with induced conformal structure $z = u + i v$ so that there is some $\omega: \Sigma \to \mathbb{R}$ with  
    \begin{equation}\label{eqn:conformal}
        \dif{s}^2 = e^{2 \omega}(\dif{u}^2 + \dif{v}^2) = e^{2 \omega} \dif{z} \dif{\bar{z}}.
    \end{equation}

Viewing $X$ as a codimension two surface in $\mathbb{R}^{3,1}$, every fiber of the normal bundle of $X$ has signature $(1,1)$ with $\mathfrak{p}$ being a constant null section of the normal bundle.
Therefore, we may find a unique $n : \Sigma \to \mathcal{L}$ so that
    \[
        \langle \dif{X}, n \rangle = 0 \quad\text{and}\quad \langle n, \mathfrak{p} \rangle = 1.
    \]
Such $n$ is called the \emph{lightlike Gauss map} of $X$, and can be used to calculate the second fundamental form via
    \[
        \mathrm{II} = L \dif{u}^2 + 2M \dif{u}\dif{v} + N \dif{v}^2
    \]
with
    \[
        L := \langle X_{uu}, n \rangle,\quad M := \langle X_{uv}, n \rangle,\quad N := \langle X_{vv}, n \rangle.
    \]
Therefore, the mean curvature $H$ is given by
    \[
        H = \frac{1}{2}e^{-2\omega}(L+N) = \frac{1}{2}e^{-2\omega} \langle X_{uu} + X_{vv}, n \rangle = 2 e^{-2 \omega} \langle X_{z \bar{z}}, n \rangle,
    \]
while the Hopf differential is given by
    \[
        Q \dif{z}^2 = \frac{1}{4}(L - N - 2 i M) \dif{z}^2 = \langle X_{zz}, n \rangle  \dif{z}^2.
    \]

The Gauss-Weingarten equation reads
    \[\begin{cases}
        X_{zz} = 2 \omega_z X_z + Q \mathfrak{p}\\
        X_{z\bar{z}} = \frac{1}{2}e^{2\omega}H \mathfrak{p}\\
        n_z = -H X_z - 2 e^{-2 \omega} Q X_{\bar{z}}
    \end{cases}\]
so that the Gauss-Codazzi equations are
    \[\begin{cases}
        \omega_{z\bar{z}}=0,&\text{(Gauss equation)}\\
        H_z = 2 e^{-2\omega} Q_{\bar{z}}. &\text{(Codazzi equation)}
    \end{cases}\]
    
Now assume that $X$ has zero mean curvature so that $H \equiv 0$.
Then the Codazzi equation implies that the Hopf differential factor is holomorphic.
As we are interested in  zero mean curvature surfaces with planar curvature lines, we reparametrize as in \cite[Lemma 1.1]{bobenko_painleve_2000} to assume without loss of generality that $Q = -\frac{1}{2}$ so that $(u,v)$ are conformal curvature line coordinates, or \emph{isothermic coordinates}.
Here, we are assuming that there are no umbilics in $\Sigma$.

The Gauss-Weingarten equations then simplify to
    \begin{equation}\label{eqn:GW}
        \begin{cases}
        X_{uu} = \omega_u X_u - \omega_v X_v - \mathfrak{p}\\
        X_{uv} = \omega_v X_u + \omega_u X_v \\
        X_{vv} = -\omega_u X_u + \omega_v X_v + \mathfrak{p}\\
        n_u = e^{-2\omega} X_u\\
        n_v = - e^{-2\omega} X_v,
        \end{cases}
    \end{equation}
while the compatibility equation becomes
    \[
        \omega_{uu} + \omega_{vv} = 0.
    \]
\section{Classification of  zero mean curvature surfaces with planar curvature lines}\label{sect:three}
In this section, we will obtain an analytic classification for  zero mean curvature surfaces with planar curvature lines using the techniques from \cite{abresch_constant_1987, akamine_analysis_2020, cho_deformation_2017, cho_deformation_2018, walter_explicit_1987}.
Then using the Weierstrass-type representation for  zero mean curvature surfaces in isotropic space \cite[\S 88]{strubecker_differentialgeometrie_1942-1}, we will recover a canonical Weierstrass data for these surfaces.
\subsection{Analytic classification} 
We first calculate the condition on $\omega$ for a  zero mean curvature surface $X$ to have planar curvature lines.
\begin{lemma}
    For an umbilic-free  zero mean curvature surface in isotropic space, the following are equivalent:
        \begin{enumerate}
            \item $u$-curvature lines are planar.
            \item $v$-curvature lines are planar.
            \item $\omega_{uv} + \omega_u \omega_v = 0$.
        \end{enumerate}
\end{lemma}
\begin{proof}
    We have that $u$-curvature lines are planar if and only if $X_u, X_{uu}, X_{uuu}$ are linearly dependent.
    Calculating that
        \[
            X_{uuu} = A X_u + (-2\omega_u \omega_v - \omega_{uv}) X_v - \omega_u \mathfrak{p}
        \]
    for some function $A$, we have that $X_u, X_{uu}, X_{uuu}$ are linearly dependent if and only if
        \[
            0 = -\begin{vmatrix}
                    1 & 0 & 0 \\
                    \omega_u & -\omega_v & -1 \\
                    A & -2\omega_u \omega_v - \omega_{uv} & - \omega_u
                \end{vmatrix}
                = \omega_{uv} + \omega_u\omega_v.
        \]
    One can show similarly that $v$-curvature lines are planar if and only if $\omega_{uv} + \omega_u\omega_v = 0$.
\end{proof}

Thus, every  zero mean curvature surface with planar curvature lines corresponds to a solution of the following system of partial differential equations for $\omega$:
    \begin{subnumcases}{\label{eqn:pde}}
        \omega_{uu} + \omega_{vv} = 0, &\text{(compatibility condition)} \label{eqn:pde1}\\
        \omega_{uv} + \omega_u \omega_v = 0. &\text{(zero mean curvature and planarity condition)} \label{eqn:pde2}
    \end{subnumcases}
Note that $\omega \equiv c$ for some real constant $c$ is a trivial solution to \eqref{eqn:pde}; to obtain non-trivial solutions to the system of partial differential equations \eqref{eqn:pde}, we reduce the problem to a system of ordinary differential equations:
\begin{proposition}
    The non-trivial solutions to \eqref{eqn:pde} are given by
        \begin{equation}\label{eqn:omega}
            e^{\omega(u,v)} = \frac{f(u)^2 + g(v)^2}{f_u(u) + g_v(v)},
        \end{equation}
    where $f(u)$ and $g(v)$ are single-variable functions such that
        \begin{subnumcases} {\label{eqn:fandg}}
            f(u) := \omega_u e^{\omega} \label{eqn:f}\\
            g(v) := \omega_v e^{\omega} \label{eqn:g}
        \end{subnumcases}
    satisfying the system of ordinary differential equations
        \begin{subnumcases}{\label{eqn:ode}}
            f_{uu} = a f \label{eqn:ode1}\\
            f_u^2 = a f^2 + b \label{eqn:ode2}\\
            g_{vv} = - a g \label{eqn:ode3}\\
            g_v^2 = -a g^2 + b \label{eqn:ode4}
        \end{subnumcases}
    for some real constants $a,b$ such that $a^2 + b^2 \neq 0$.
\end{proposition}
\begin{proof}
    Since $\omega \equiv c$ is a trivial solution to \eqref{eqn:pde}, we assume that $\omega \not \equiv c$.
    From \eqref{eqn:pde2}, we see that
        \[
            \frac{\omega_{uv}}{\omega_v} = -\omega_u
        \]
    so that integrating both sides with respect to $u$ gives
        \[
            \log \omega_v = -\omega + k_1(v)
        \]
    for some function $k_1(v)$.
    Therefore, we have
        \[
            g(v) := e^{k_1(v)} = \omega_v e^\omega.
        \]
    Similarly, \eqref{eqn:pde2} also implies that
        \[
            f(u) := e^{k_2(u)} = \omega_u e^{\omega}
        \]
    for some function $k_2(u)$.
    Thus, we have by \eqref{eqn:pde1} that
        \begin{align*}
            0 &= \omega_{uu} + \omega_{vv} = (f e^{-\omega})_u +  (g e^{-\omega})_v \\
                &= f_u e^{-\omega} - f \omega_u e^{-\omega} + g_v e^{-\omega} - g \omega_v e^{-\omega} = e^{-\omega}(f_u - f \omega_u + g_v - g \omega_v),
        \end{align*}
    which in turn tells us
        \[
            0 = f_u - f \omega_u + g_v - g \omega_v = f_u - f^2 e^{-\omega} + g_v - g^2 e^{-\omega}
        \]
    giving us \eqref{eqn:omega}.

    Now we will show that $f$ and $g$ satisfies \eqref{eqn:ode}.
    Taking the $u$-derivative of \eqref{eqn:omega}, we see by \eqref{eqn:f} that
        \[
           f = \omega_u e^{\omega} = \frac{2ff_u (f_u + g_v) - (f^2+g^2)f_{uu}}{(f_u + g_v)^2}
        \]
    so that
        \begin{align*}
            f (f_u^2 - g_v^2) - (f^2+g^2)f_{uu} = 0.
        \end{align*}
    Thus, we have
        \begin{equation}\label{eqn:ode1almost}
            f_{uu} = a f
        \end{equation}
    where
        \begin{equation}\label{eqn:adef}
            a := \frac{f_u^2 - g_v^2}{f^2+g^2}.
        \end{equation}
    Note that since $f_{uu}$ and $f$ are functions of $u$ alone, \eqref{eqn:ode1almost} implies $a_v = 0$.
    
    On the other hand, taking the $v$-derivative of \eqref{eqn:omega} and using \eqref{eqn:g}, we have
        \[
            g = \omega_v e^{\omega} = \frac{2g g_v (f_u + g_v) - (f^2 + g^2)g_{vv}}{(f_u + g_v)^2},
        \]
    which implies
        \[
            -g (f_u^2 - g_v^2) - (f^2 + g^2)g_{vv} = 0.
        \]
    Therefore, we have by \eqref{eqn:adef} that
        \[
            g_{vv} = - \frac{f_u^2 - g_v^2}{f^2 + g^2} g = -a g,
        \]
    but since $g_{vv}$ and $g$ are functions of $v$ alone, we have $a_u = 0$.
    Thus we conclude that $a$ is some real constant, giving us \eqref{eqn:ode1} and \eqref{eqn:ode3}.

    Finally, multiplying \eqref{eqn:ode1} by $2f_u$ and integrating by $u$ gives us
        \[
            f_u^2 = a f^2 + b,
        \]
    while multiplying \eqref{eqn:ode3} by $2g_v$ and integrating by $v$ implies
        \[
            g_v^2 = - a g^2 + \tilde{b}
        \]
    for some constants of integration $b$ and $\tilde{b}$.
    Thus by \eqref{eqn:adef}, we have
        \[
            b - \tilde{b} = f_u^2 - g_v^2 - a(f^2 + g^2) = 0,
        \]
    giving us \eqref{eqn:ode2} and \eqref{eqn:ode4}.
\end{proof}

Before we solve \eqref{eqn:ode} explicitly, we derive an additional condition for $b$:
\begin{lemma}\label{lemma:bcond}
	The constant $b$ in \eqref{eqn:ode} is non-negative, that is, $b \geq 0$.
\end{lemma}
\begin{proof}
	Suppose for contradiction that $b < 0$.
    If $a \geq 0$, then \eqref{eqn:ode4} gives us
        \[
            g_v^2 = -a g^2 + b < 0,
        \]
    which is a contradiction.
	If $a < 0$, then \eqref{eqn:ode2} implies 
        \[
            f_u^2 = a f^2 + b < 0,
        \]
    which is also a contradiction.
\end{proof}
We also note here that by switching the parameters $(u,v)$, the system of ordinary differential equations \eqref{eqn:ode} are symmetric in $f$ and $g$: thus, we may assume without loss of generality that $a \geq 0$.
Under these conditions of $a \geq 0$ and $b \geq 0$, let us take the constants $\alpha := \sqrt{a}$ and $\beta := \sqrt{b}$, and in the case $\alpha \neq 0$, we can solve \eqref{eqn:ode1} and \eqref{eqn:ode3} to obtain the general solutions
    \begin{subnumcases}{\label{eqn:gen}}
         f = C_1 e^{\alpha u} + C_2 e^{-\alpha u} \label{eqn:gen1}\\
         g = C_3 e^{i \alpha v} + C_4 e^{-i \alpha v} \label{eqn:gen2}
    \end{subnumcases}
where $C_1$, $C_2$, $C_3$, and $C_4$ are constants such that $C_1, C_2 \in \mathbb{R}$ and $C_3 = \bar{C}_4 \in \mathbb{C}$ since $f$ and $g$ are real-valued functions.
Then using \eqref{eqn:ode2} and \eqref{eqn:ode4}, we obtain additional relations on the constants:
    \begin{equation}\label{eqn:cicond}
        C_1 C_2 = - \frac{\beta^2}{4 \alpha^2} = - C_3 C_4 = - |C_3|^2.
    \end{equation}
Using these relations, we can identify the initial conditions for $f$ and $g$ using the following lemma:

\begin{lemma}\label{lemma:init}
	The solutions $f$ (resp. $g$) to \eqref{eqn:ode} has a zero if and only if $b>0$ or $f \equiv 0$ (resp. $g \equiv 0$).
\end{lemma}
\begin{proof}
	To show one direction, suppose that there exists some $u_0$ such that $f(u_0)=0$.
    Using Lemma~\ref{lemma:bcond} we know that $b \geq 0$; thus, we have only to see that if $b=0$ then $f \equiv 0$.
    When $b=0$ so that $\beta = 0$ (and $\alpha >0$), \eqref{eqn:cicond} implies that
        \[
            C_1 C_2 = 0.
        \]
    Assuming without loss of generality that $C_2 = 0$, we have $f = C_1 e^{\alpha u}$ which has a zero only if $C_1 = 0$, that is, $f \equiv 0$.
    The statement for $g$ is proven similarly.

	To show the other direction, we have only to see that if $b>0$ then $f$ and $g$ has a zero.
    So let $f$ be of the form \eqref{eqn:gen1} satisfying \eqref{eqn:cicond} with $\beta > 0$.
    Then if we let
        \[
             u_0 := \frac{1}{2\alpha}\log\left(\frac{\beta^2}{4 C_1^2 \alpha^2}\right) \in \mathbb{R},
        \]
    we can check directly that $f(u_0) = 0$.
    
    On the other hand, if $g$ takes the form \eqref{eqn:gen2} satisfying \eqref{eqn:cicond} with $\beta > 0$, then we may write
        \[
            \bar{C}_4 = C_3 = \frac{\beta}{2 \alpha}e^{i \theta}
        \]
    for some $\theta \in \mathbb{R}$.
    Then if we have
        \[
            v_0 = \frac{1}{2 \alpha}(\pi - 2 \theta),
        \]
    then we can verify that $g(v_0) = 0$.
\end{proof}

Now we are in a position to solve for $f$ and $g$ explicitly using \eqref{eqn:gen} with \eqref{eqn:cicond}.

\textbf{Case (a)}: Let $\alpha > 0$ and $\beta = 0$. Then we have by \eqref{eqn:cicond} that
    \[
        C_1 C_2 = 0 = |C_3|^2
    \]
so that we obtain 
    \[\begin{cases}
        f = C e^{\pm \alpha u}\\
        g = 0
    \end{cases}\]
for some nonzero real constant $C$.
By considering coordinate change in $u \mapsto -u$, we may assume without loss of generality that $f = C e^{\alpha u}$ which in turn implies that $C > 0$ since
    \[
        0 < e^\omega = \frac{f^2 + g^2}{f_u + g_v} = \frac{C}{\alpha} e^{- \alpha u}.
    \]
Then a parameter change in the form of $u \mapsto -u - \frac{1}{\alpha}\log C$, allows us to simplify so that
    \[\begin{cases}
        f = e^{-\alpha u}\\
        g = 0.
    \end{cases}\]

\textbf{Case (b)}: If $\alpha = 0$ and $\beta > 0$, then \eqref{eqn:ode} becomes
        \[\begin{cases}
            f_{uu} = 0\\
            f_u^2 = \beta^2
        \end{cases} \quad\text{and}\quad
        \begin{cases}
            g_{vv} = 0\\
            g_v^2 = \beta^2
        \end{cases},\]
and using Lemma~\ref{lemma:init}, we have that
    \[\begin{cases}
        f = \pm \beta u\\
        g = \pm \beta v.
    \end{cases}\]
Considering coordinate changes and the fact that $f_u + g_v \neq 0$ from \eqref{eqn:omega}, we may assume without loss of generality that
    \[\begin{cases}
        f = \beta u\\
        g = \beta v.
    \end{cases}\]

\textbf{Case (c)}: Finally, if $\alpha > 0$ and $\beta > 0$, then using Lemma~\ref{lemma:init} on general solutions \eqref{eqn:gen}, we may assume $f(0) = g(0) = 0$, so that
    \[
        0 = C_1 + C_2 = C_3 + C_4.
    \]
Thus, \eqref{eqn:cicond} tells us
    \[
        C_1 = \pm \frac{\beta}{2\alpha} \quad\text{and}\quad C_3 = \pm \frac{i \beta}{2\alpha},
    \]
and the solutions are given by
    \[\begin{cases}
        f = \pm \frac{\beta}{\alpha} \sinh \alpha u\\
        g = \pm \frac{\beta}{\alpha} \sin \alpha v.
    \end{cases}\]
Again, by considering coordinate changes and shift in parameters $v \mapsto v + \frac{\pi}{\alpha}$, we may assume
    \[\begin{cases}
        f = \frac{\beta}{\alpha} \sinh \alpha u\\
        g = \frac{\beta}{\alpha} \sin \alpha v.
    \end{cases}\]

We summarize the discussions in the section in the following classification of all zero mean curvature surfaces with planar curvature lines (see also Figure~\ref{fig:classification}):
\begin{theorem}\label{thm:analytic}
    Let $X : \Sigma \to \mathbb{I}^3$ be a non-planar zero mean curvature immersion conformally parametrized by $(u,v) \in \Sigma$ so that
        \[
            \dif{s}^2 = e^{2\omega}(\dif{u}^2 + \dif{v}^2)
        \]
    with normalized Hopf differential factor $Q \dif{z}^2 = -\frac{1}{2} \dif{z}^2$.
    Then $X$ has planar curvature lines if and only if $\omega$ satisfies one of the following cases:
    \begin{description}
        \item[Case (1)] If $\omega$ is not constant,
        \item[Case (1a)] For $\alpha > 0$ and $\beta = 0$, we have
            \[\begin{cases}
                f = e^{-\alpha u}\\
                g = 0
                \end{cases}\]
            so that
            \[
                e^\omega = -\frac{1}{\alpha}e^{-\alpha u},
            \]
        and thus is a surface of revolution.
        \item[Case (1b)] For $\alpha = 0$ and $\beta > 0$, we have
            \[\begin{cases}
                f = \beta u\\
                g = \beta v,
            \end{cases}\]
            implying
            \[
                e^\omega = \frac{\beta}{2}(u^2 + v^2),
            \]
        and thus is non-periodic in both the $u$-direction and the $v$-direction.
        \item[Case (1c)] For $\alpha > 0$ and $\beta > 0$, we have
            \[\begin{cases}
                f = \frac{\beta}{\alpha} \sinh \alpha u\\
                g = \frac{\beta}{\alpha} \sin \alpha v,
            \end{cases}\]
            giving rise to
            \[
                e^\omega = \frac{\beta}{\alpha^2}(\cosh \alpha u - \cos \alpha v),
            \]
        and thus is non-periodic in the $u$-direction, but periodic in the $v$-direction.
        \item[Case (2)] We have that $\omega \equiv c$ for some $c \in \mathbb{R}$.
    \end{description}
\end{theorem}

\begin{figure}
    \includegraphics[width=0.4\linewidth]{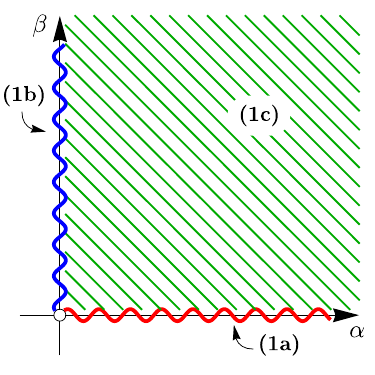}
    \caption{Classification diagram of zero mean curvature surfaces with planar curvature lines}
    \label{fig:classification}
\end{figure}

\begin{remark}
    Note that we have dropped the umbilic-free condition as the solutions to \eqref{eqn:pde} as given in Theorem~\ref{thm:analytic} extends globally.
\end{remark}

\subsection{Recovering the Weierstrass data}
As in the cases of zero mean curvatures surfaces in Euclidean space and Minkowski $3$-space, zero mean curvature surfaces in isotropic $3$-space also enjoy a \emph{Weierstrass-type representation}, given as follows:
\begin{fact}{\cite[\S 88]{strubecker_differentialgeometrie_1942-1}}
    Any zero mean curvature surface $X : \Sigma \to \mathbb{I}^3$ over a simply-connected domain $\Sigma$ can be locally be represented as
        \[
            X = \Re \int (h, 1, -i) \eta
        \]
    for some meromorphic function $h$ and holomorphic $1$-form $\eta \dif{z}$ with holomorphic $h^2 \eta$.
    Then the zero mean curvature surface has induced metric
        \[
            \dif{s}^2 = |\eta|^2 
        \]
    with the Hopf differential
        \[
            Q \dif{z}^2 = \frac{1}{2} \eta \dif{h}
        \]
    while the lightlike Gauss map is given by
        \[
            n = -\frac{1}{2}(1 + |h|^2, 2 \Re h, -2 \im h, -1 + |h|^2).
        \]
    We call $(h, \eta)$ the Weierstrass data.
\end{fact}
Furthermore, the conjugate zero mean curvature surface of a given zero mean curvature surface can be defined as follows using Weierstrass data:
\begin{definition}{\cite[\S 89]{strubecker_differentialgeometrie_1942-1}}
    Let $X : \Sigma \to \mathbb{I}^3$ be a zero mean curvature surface represented by the Weierstrass data $(h, \eta)$.
    The zero mean curvature surface $\tilde{X} : \Sigma \to \mathbb{I}^3$ given by the Weierstrass data $(h, i\eta)$ is called the \emph{conjugate zero mean curvature surface} of $X$.
    The pair of surfaces $X$ and $\tilde{X}$ are referred to as \emph{conjugate zero mean curvature pair}.
\end{definition}
We make two important remarks regarding Weierstrass representations:
\begin{remark}\label{remark:wdata}
    The Weierstrass data of a zero mean curvature surface change as isometries are applied to the surface.
    Therefore, the Weierstrass data are not uniquely determined from the intrinsic data of a given zero mean curvature surface.
    However, if different Weierstrass data give the same intrinsic data of the zero mean curvature surfaces then the two surfaces are still congruent up to isometries of isotropic $3$-space.
\end{remark}
\begin{remark}\label{remark:conjugate}
    We note that the conformal factor $\omega$ and the Hopf differential $Q \dif{z}^2$ are given by the Weierstrass data.
    Thus, if a pair of zero mean curvature surfaces share the same conformal factor, while the Hopf differential differs by a factor of $i$, then we can deduce that the two zero mean curvature surfaces are a conjugate zero mean curvature pair, up to isometries of isotropic space.
\end{remark}

Under the current setting of normalizing $Q = -\frac{1}{2}$, we have that the holomorphic $1$-form in the Weierstrass representation must take the form
    \begin{equation}\label{eqn:etatoh}
        \eta = - \frac{1}{h_z} \dif{z}.
    \end{equation}
To recover the Weierstrass data from the conformal factor, we will use the following result \cite[\S 4]{shaw_recovering_2004}:
\begin{fact}\label{fact:recover}
    Let $\eta: \Sigma \to \mathbb{C}$ be a holomorphic function with some given modulus $|\eta|^2 = R^2$.
    Then the holomorphic function can be recovered via
        \[
            \eta(z) = \frac{1}{\overline{\eta(z_0)}}R^2\left(\frac{z+\bar{z_0}}{2},\frac{z-\bar{z_0}}{2 i}\right)
        \]
    for some $z_0\in \Sigma$.
\end{fact}
\begin{remark}\label{remark:recover}
    The holomorphic function recovered from its modulus via Fact~\ref{fact:recover} is not unique, as one can multiply the holomorphic function by a unit length complex constant.
    This freedom can be observed in Fact~\ref{fact:recover} in the choice of $\eta(z_0)$, as 
        \[
            |\eta(z_0)|^2 = R^2 (\Re (z_0), \im(z_0)),
        \]
    and thus
        \[
            \eta(z_0) = e^{i \theta} R(\Re (z_0), \im(z_0))
        \]
    for some $\theta \in \mathbb{R}$.
    However, by Remark~\ref{remark:wdata}, we are free to choose any $\theta$ in recovering the Weierstrass data.
\end{remark}
Using Fact~\ref{fact:recover}, Remark~\ref{remark:recover}, and \eqref{eqn:etatoh}, we recover the Weierstrass data as follows.

\textbf{Case (1a):} Let $\alpha > 0$ and $\beta = 0$.
In this case, the metric is
    \[
        |\eta_{(\alpha, 0)}|^2 = \frac{1}{\alpha^2}e^{-2\alpha u} |{\dif{z}}|^2.
    \]
Therefore, we may take
    \[
        \eta_{(\alpha, 0)} = -\frac{1}{\alpha}e^{-\alpha z} \dif{z},
    \]
so that by \eqref{eqn:etatoh}, we have
    \[
        h_{(\alpha, 0)}=e^{\alpha z}.
    \]

\textbf{Case (1b):} If we have $\alpha = 0$ and $\beta > 0$, then the metric is given by
    \[
        |\eta_{(0, \beta)}|^2 = \frac{\beta^2}{4}(u^2 + v^2)^2|{\dif{z}}|^2,
    \]
so that we may take
    \[
        \eta_{(0,\beta)} = \frac{\beta z^2}{2} \dif{z},
    \]
and thus
    \[
        h_{(0,\beta)} = \frac{2}{\beta z}.
    \]

\textbf{Case (1c):} Now, let $\alpha > 0$ and $\beta > 0$, so that the metric becomes
    \[
        |\eta_{(\alpha, \beta)}|^2 = \frac{\beta^2}{\alpha^4}(\cosh \alpha u - \cos \alpha v)^2 |{\dif{z}}|^2.
    \]
Thus, we take
    \[
        \eta_{(\alpha,\beta)} = \frac{2\beta}{\alpha^2} \sinh^2 \frac{\alpha z}{2} \dif{z},
    \]
so that
    \[
        h_{(\alpha,\beta)} =\frac{\alpha}{\beta} \coth \frac{\alpha z}{2}.
    \]

\textbf{Case (2):} Finally, let $\omega \equiv -c$ for some constant $c \in \mathbb{R}$.
Then the metric is
    \[
        |\eta_c|^2 = e^{-2c}
    \]
so that we may take
    \[
        \eta_c = -e^{-c} \dif{z},
    \]
implying that
    \[
        h_c = e^{c} z.
    \]

Noting that the change of coordinates $\tilde{h}(z) := h (Cz)$ with $\tilde{\eta}$ given by \eqref{eqn:etatoh} corresponds to the surface changing by
    \[
        \tilde{X}(u,v) = \frac{1}{C^2}X(Cu, Cv),
    \]
so that it is a reparametrization of the surface up to some homothety factor, we summarize our classification of zero mean curvature surfaces given with their Weierstrass data (see also Figure~\ref{fig:examples}):
\begin{theorem}\label{thm:wdata}
    Let $X : \Sigma \to \mathbb{I}^3$ be a zero mean curvature immersion with planar curvature lines.
    Then $X$ must be a piece of one, and only one, of
        \begin{itemize}
            \item plane $(0, 1 \dif{z})$,
            \item trivial Enneper-type surface $(z, 1\dif{z})$,
            \item catenoid $(e^{z}, e^{-z} \dif{z})$,
            \item Enneper-type surface $(\frac{1}{z}, z^2 \dif{z})$, or
            \item one of Bonnet-type surfaces $(\frac{\alpha}{\beta} \coth \frac{\alpha z}{2}, \frac{2\beta}{\alpha^2} \sinh^2 \frac{\alpha z}{2} \dif{z})$,
        \end{itemize}
    given with the respective Weierstrass data, up to isometries and homotheties of the isotropic space.
\end{theorem}

\begin{figure}
    \begin{minipage}{0.49\linewidth}
    	\centering
	\includegraphics[width=0.8\textwidth]{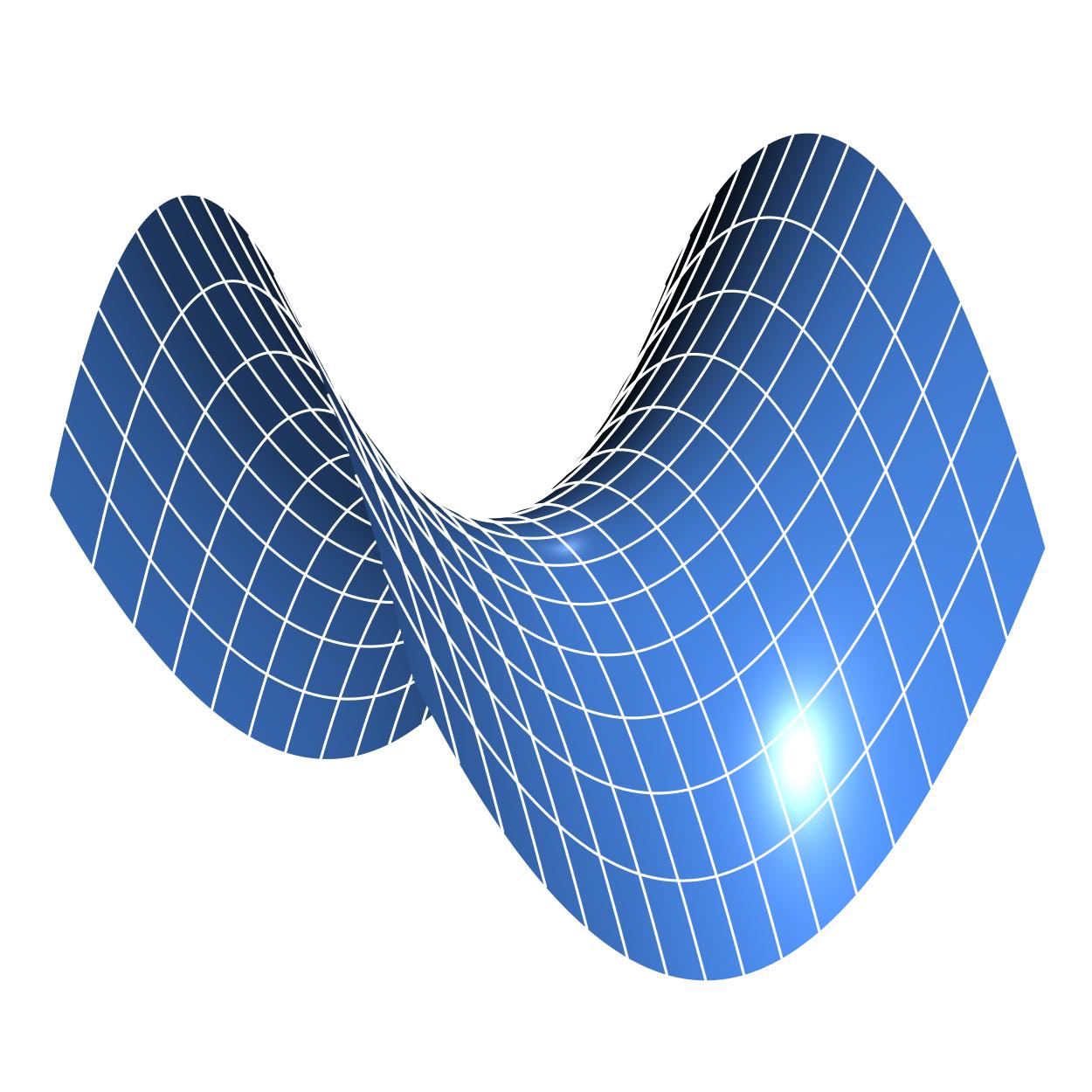}
    \end{minipage}
    \begin{minipage}{0.49\linewidth}
    	\centering
	\includegraphics[width=0.8\textwidth]{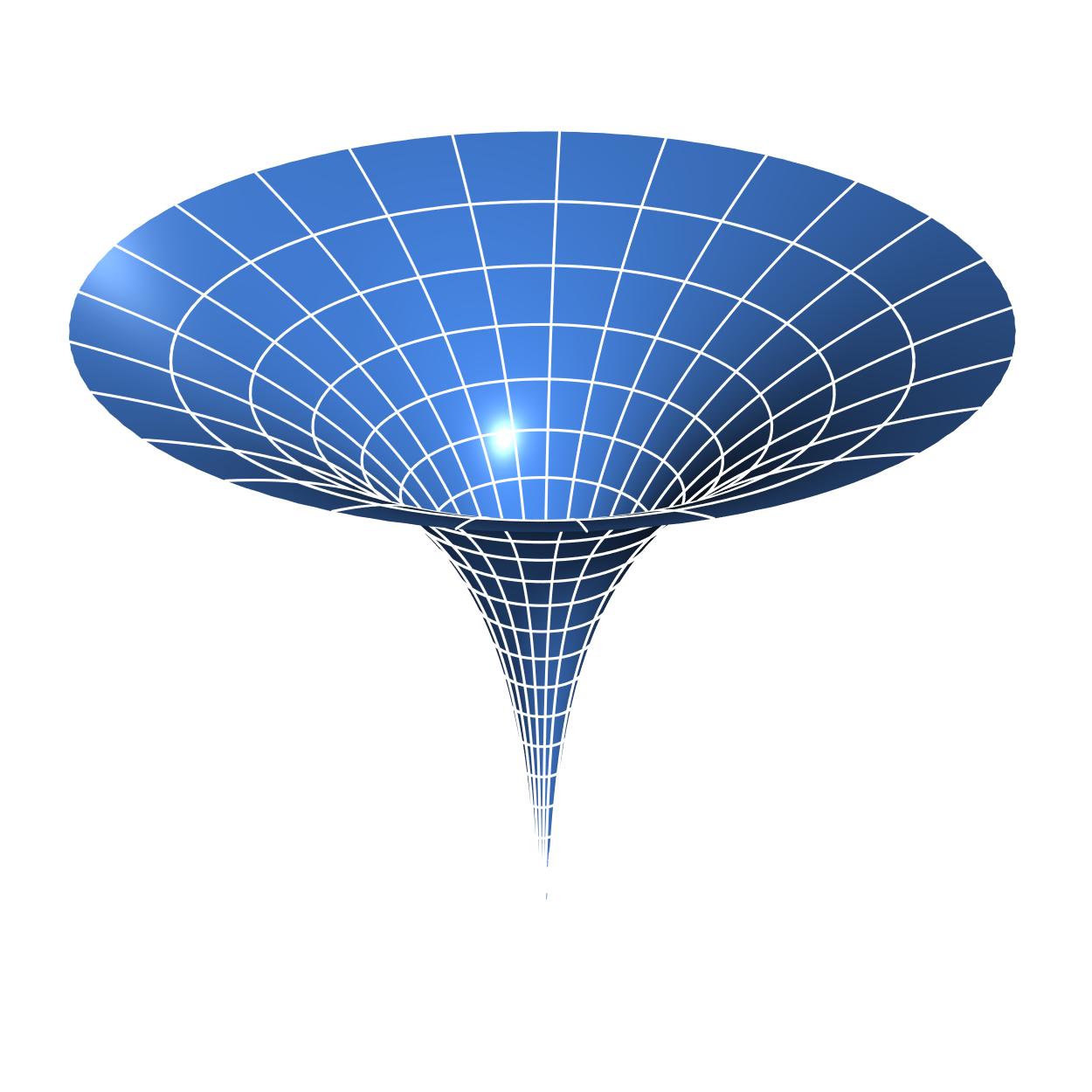}
    \end{minipage}
    
    \vspace{5pt}
    \begin{minipage}{0.49\linewidth}
    	\centering
	\includegraphics[width=0.8\textwidth]{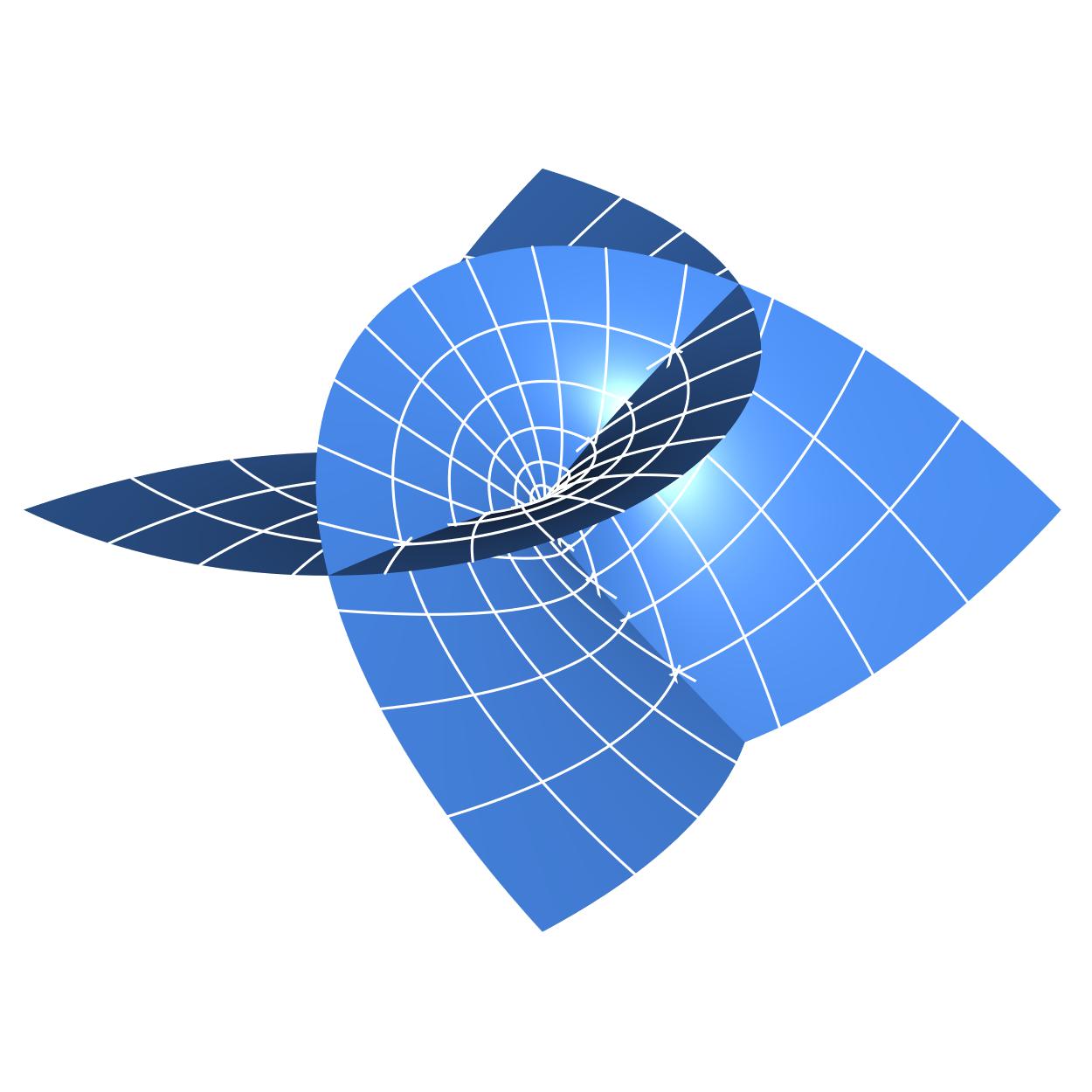}
    \end{minipage}
    \begin{minipage}{0.49\linewidth}
    	\centering
	\includegraphics[width=0.8\textwidth]{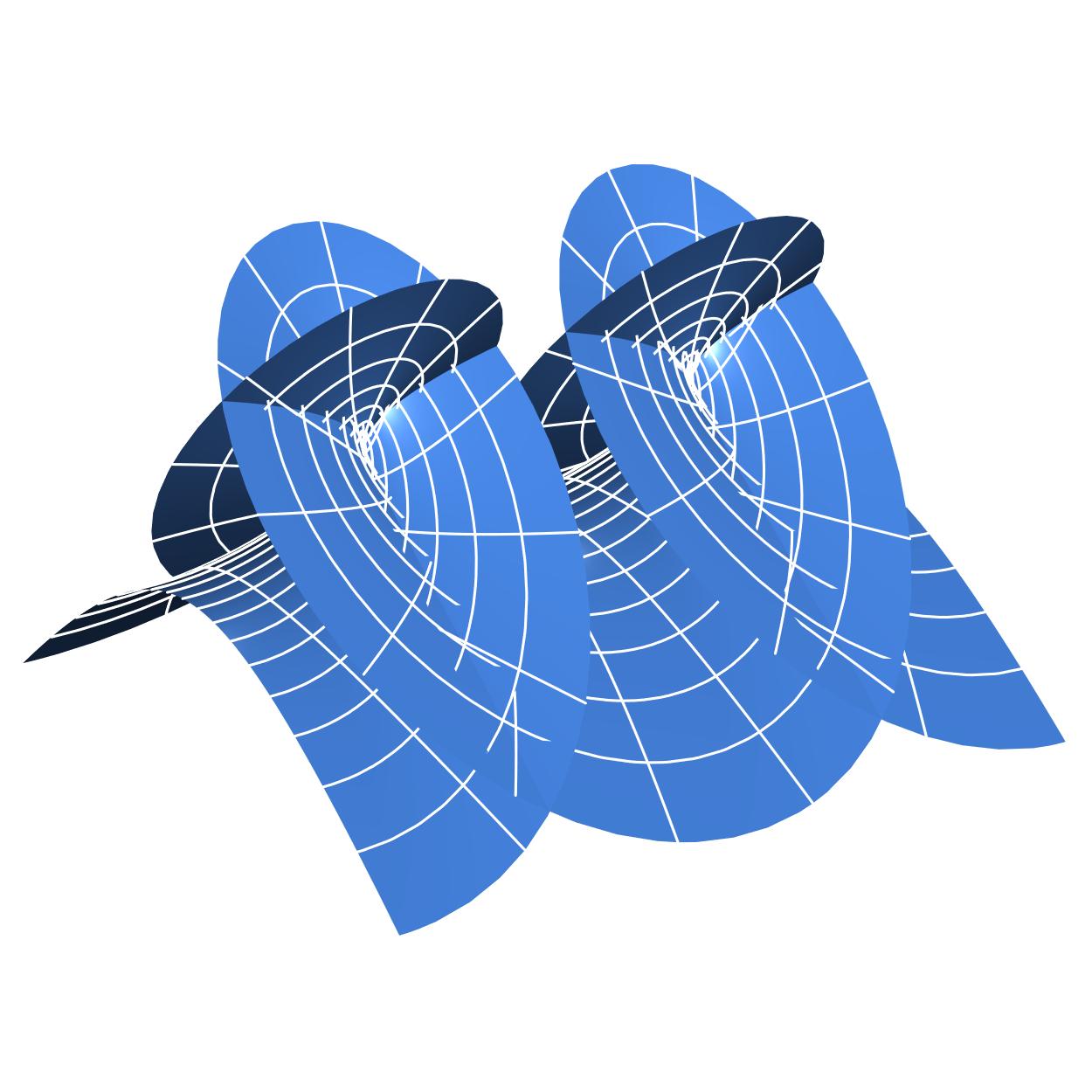}
    \end{minipage}

    \caption{Zero mean curvature surfaces with planar curvature lines in isotropic space: trivial Enneper surface and catenoid on the top row; Enneper surface and Bonnet-type surface on the bottom row.}
    \label{fig:examples}
\end{figure}

\section{Geometry of zero mean curvature surfaces with planar curvature lines}\label{sect:four}
Here, we note a few geometric facts about zero mean curvature surfaces with planar curvature lines.
\subsection{Trivial Enneper-type surface}
The zero mean curvature surface pertaining to the trivial solution to the integrability condition, that is, $\omega \equiv c$, has been studied in the context of higher order Enneper-type surfaces in isotropic $3$-space in \cite[\S 9]{strubecker_uber_1975}, and appears as the trivial example of such surfaces.
Assuming without loss of generality that $c = 0$, the parametrization of the surface reads
    \[
        X_0 (u,v) = -\left(\tfrac{1}{2}(u^2 - v^2), u, v \right)^t
    \]
so that every curvature line of the surface is in the shape of a parabola in Euclidean geometry.
These are the \emph{parabolic circles} \cite[\S 11]{strubecker_differentialgeometrie_1941} of isotropic $3$-space.
Thus, every curvature line of the trivial Enneper surface is a circle in isotropic $3$-space, so that the surface is a doubly-channel surface.

In particular, the parametrization also shows that all the (parabolic) circles in a family of curvature lines belong to planes that are parallel to each other.
This alludes to the fact that the surface may be invariant under certain actions of isometries of isotropic $3$-space.

Isometries of isotropic $3$-space that fix the origin are inherited from the isometries of $\mathbb{R}^{3,1}$ via
    \[
        \mathrm{SO}(2,0,1) := \{ A \in \mathrm{SO}(3,1) : A \mathfrak{p} = \mathfrak{p}\}
    \]
acting on isotropic $3$-space.
To view parabolic circles as orbits under action of isometries, let us consider the subgroup $P_\mathbf{v} < \mathrm{SO}(2,0,1)$ given by
    \[
        P_\mathbf{v} := \{A \in \mathrm{SO}(2,0,1) : A \mathbf{v} = \mathbf{v}, \langle \mathbf{v}, \mathbf{v} \rangle = 1\}.
    \]
Normalizing so that $\mathbf{v} = (0,1,0,0)^t =: e_1$, the subgroup $P_\mathbf{v}$ can be parametrized via
    \[
        v \mapsto \begin{pmatrix}
            1 + \tfrac{v^2}{2} & 0 & v & -\tfrac{v^2}{2} \\
            0 & 1 & 0 & 0 \\
            -v & 0 & -1 & v \\
            \tfrac{v^2}{2} & 0 & v & 1-\tfrac{v^2}{2}
        \end{pmatrix} =: P_{e_1}(v).
    \]
    
Now consider the following (well-defined) action on $\mathbb{I}^3$:
    \[
        X \mapsto P_{e_1,r}(v) X := P_{e_1}(v)(X + r \tilde{\mathfrak{p}}) - r \tilde{\mathfrak{p}}.
    \]
This action acts as an isometry of $\mathbb{I}^3$; furthermore, applying this action to the origin $O = (0,0,0,0)^t$, we see that
    \[
        P_{e_1,r}(v)(O) = (-\tfrac{r}{2} v^2, 0, r v, -\tfrac{r}{2}v^2)^t,
    \]
that is, parabolic circles are obtained as orbits under the action of $P_{\mathbf{v},r}$.
We refer to such actions as \emph{parabolic rotations}.

Defining $e_2 := (0,0,1,0)^t$, we can now similarly check that $P_{e_2}$ can be parametrized via
    \[
        u \mapsto \begin{pmatrix}
            1 + \tfrac{u^2}{2} & -u & 0 & -\tfrac{u^2}{2} \\
            u & -1 & 0 & -u \\
            0 & 0 & 1 & 0 \\
            \tfrac{u^2}{2} & -u & 0 & 1-\tfrac{u^2}{2}
        \end{pmatrix} =: P_{e_2}(u),
    \]
and calculate that
    \[
        P_{e_2,1}(u)\circ P_{e_1,-1}(v) (O) = -\left(\tfrac{1}{2}(u^2 - v^2), u, v, \tfrac{1}{2}(u^2 - v^2)\right)^t.
    \]
Therefore, we not only observe that the trivial Enneper-type surface is invariant under parabolic rotations, but also conclude that the surface is generated by applying two of such rotations to a single point.

\subsection{Axial directions}
As in the case of zero mean curvature surfaces in Euclidean space and Minkowski space \cite{akamine_analysis_2020, cho_deformation_2017, cho_deformation_2018}, we show that zero mean curvature surfaces with planar curvature lines in isotropic $3$-space also have \emph{axial directions} following the techniques from \cite{walter_explicit_1987}.
Axial direction is the unique direction contained in the planes of a family of planar curvature lines, that is, there is some constant $w_1, w_2 \in \mathbb{I}^3$ such that
    \[
        w_1 \in \Span\{X_u, X_{uu}\} \quad\text{and}\quad w_2 \in \Span\{X_v, X_{vv}\}
    \]
for all $(u,v)\in \Sigma$.
However, unlike the cases of zero mean curvature surfaces in Euclidean space or Minkowski space, the difficulty of finding this direction lies in the fact that there is no cross product structure in isotropic $3$-space.

To obtain a characterization of a direction $w \in \mathbb{I}^3$ being contained in a plane $P \subset \mathbb{I}^3$, recall from \cite[\S 2.1.1]{cho_spinor_2024} that planes in $\mathbb{I}^3$ are given by $m \in \mathbb{R}^{3,1}$ with $\langle m, m \rangle = 0$ and $\langle m, \mathfrak{p} \rangle = 1$ via
    \[
        P_{m,q} = \{x \in \mathbb{I}^3 : \langle x, m \rangle = q \}.
    \]
When viewing the plane as a surface, $m$ is the constant lightlike Gauss map of $P_{m,q}$.
Thus, if $w \in P_{m,q}$, then we must have that
    \[
        \langle m, w \rangle = 0.
    \]

Denoting by $P_1(v)$ the plane containing the $u$-curvature line at $v$, that is,
    \[
        P_1(v) = \Span \{X_u(u,v), X_{uu}(u,v)\},
    \]
we first find the (unnormalized) lightlike Gauss map $m_1(v)$ of $P_1(v)$ by requiring that
    \[
        \langle m_1,m_1 \rangle = 0, \quad \langle m_1,X_u \rangle = 0,  \quad  \langle m_1,X_{uu} \rangle = 0,  \quad  \langle m_1,\mathfrak{p} \rangle = \omega_v,
    \]
where we have assumed $\omega_v \not\equiv 0$.
Therefore, we can find
    \[
        m_1 = -e^{-2\omega} X_v + \omega_v n - \frac{e^{-2\omega}}{2\omega_v} \mathfrak{p}.
    \]
Similarly, denoting by $P_2(u)$ the plane containing the $v$-curvature line at $u$, we can calculate the (unnormalized) lightlike Gauss map $m_2(u)$ of $P_2(u)$ as
    \[
        m_2 = e^{-2\omega} X_u + \omega_u n - \frac{e^{-2\omega}}{2\omega_u}\mathfrak{p}.
    \]
Then we have
    \begin{align*}
        m_{1,v} &= \omega_u e^{-2\omega} X_u + \omega_{vv} n + \frac{\omega_{vv}e^{-2\omega}}{2 \omega_v^2} \mathfrak{p}, \\
        m_{2,u} &= -\omega_v e^{-2\omega} X_v + \omega_{uu} n + \frac{\omega_{uu}e^{-2\omega}}{2 \omega_u^2} \mathfrak{p}.
    \end{align*}
    
Now if we define $w_1, w_2 \in \mathbb{I}^3$ via
    \begin{align*}
        w_1 &:= \omega_{uu} X_u - \omega_{uv} X_v +\omega_u \mathfrak{p}\\
        w_2 &:= \omega_{uv} X_u - \omega_{vv} X_v +\omega_v \mathfrak{p},
    \end{align*}
then we can show
    \[
        \langle m_1, w_1 \rangle = \omega_{uv} + \omega_v\omega_u = 0
        \quad\text{and}\quad
        \langle m_2, w_2 \rangle = \omega_{uv} + \omega_v\omega_u = 0,
    \]
while
    \[
        \langle m_{1,v}, w_1 \rangle = \omega_{uu} \omega_u + \omega_{vv}\omega_u = 0
        \quad\text{and}\quad
        \langle m_{2,u}, w_2 \rangle = \omega_v \omega_{vv} + \omega_v\omega_{uu} = 0.
    \]
Finally, using \eqref{eqn:pde} and their partial derivatives, we can show that
	\begin{align*}
		w_{1,u} &= (\omega_{uuu} + \omega_{uu}\omega_u - \omega_{uv}\omega_v) X_u
                    - (\omega_{uuv} + \omega_{uu}\omega_v + \omega_{uv}\omega_u) X_v = 0,\\
		w_{1,v} &= (\omega_{uuv} + \omega_{uu}\omega_v + \omega_{uv}\omega_u) X_u
                    - (\omega_{uvv} - \omega_{uu}\omega_u + \omega_{uv}\omega_v)X_v = 0, \\
        w_{2,u} &= (\omega_{uuv} + \omega_{uv}\omega_u - \omega_{vv}\omega_v) X_u
                    - (\omega_{uvv} + \omega_{uv}\omega_v + \omega_{vv}\omega_u)X_v = 0, \\
        w_{2,v} &= (\omega_{uvv} + \omega_{uv}\omega_v + \omega_{vv}\omega_u) X_u
                    - (\omega_{vvv} - \omega_{uv}\omega_u + \omega_{vv}\omega_v)X_v = 0.
	\end{align*}
Therefore, we have shown that $w_1$ and $w_2$ are constant vectors contained in the planes defined by $m_1$ and $m_2$, respectively, and we summarize as follows:
\begin{proposition}
	If $f(u)$ (resp. $g(v)$) is not identically equal to zero, then there is a unique constant direction $w_1$ (resp. $w_2$) contained in all the planes of $u$-curvature lines (resp. $v$-curvature lines), given by
	\[
		\begin{cases}
			w_1 = \omega_{uu} X_u - \omega_{uv} X_v +\omega_u \mathfrak{p} \\
            w_2 = \omega_{uv} X_u - \omega_{vv} X_v +\omega_v \mathfrak{p}.
		\end{cases}
	\]
	Furthermore, if $w_1$ and $w_2$ both exist, then $w_1$ is orthogonal to $w_2$.
    We call $w_1$ and $w_2$ the \emph{axial directions} of the  surface.
\end{proposition}
\begin{proof}
    We only need to verify that
        \[
            0 = \langle w_1, w_2 \rangle = \omega_{uv}e^{2\omega} (\omega_{uu} + \omega_{vv})
        \]
    which vanishes by \eqref{eqn:pde1}.
\end{proof}

To see the causality of the axial directions, using \eqref{eqn:fandg} tells us
	\begin{align*}
		\langle w_1,w_1 \rangle &= e^{2\omega} \omega_{uv}^2 + e^{2\omega} \omega_{uu}^2
		      = f^2 g^2 e^{-2\omega} + (f_u - f^2 e^{-\omega})^2\\
		  &= f^2 ((f^2+g^2) e^{-2\omega} - 2 f_u e^{-\omega}) + f_u^2 = -f^2\frac{f_u^2-g_v^2}{f^2+g^2}+f_u^2 \\
            &= -f^2e^{-\omega}(f_u- g_v)+ \alpha^2 f^2 + \beta^2 = \beta^2,
	\end{align*}
where we have used the fact that
    \[
        e^\omega = \frac{f^2 + g^2}{f_u + g_v}
            = \frac{(f^2 + g^2)(f_u - g_v)}{f_u^2 - g_v^2}
            = \frac{f_u - g_v}{\alpha^2}. 
    \]
coming from \eqref{eqn:ode}.
Similarly, one can calculate that
    \[
        \langle w_2, w_2 \rangle = \beta^2.
    \]
    
\section{Deformation of zero mean curvature surfaces with planar curvature lines}\label{sect:five}
We will now show that the class of zero mean curvature surfaces with planar curvature lines can be described via a continuous deformation, mirroring the result of zero mean curvature surfaces in Euclidean space and Minkowski space \cite{akamine_analysis_2020, cho_deformation_2017, cho_deformation_2018}.
In this paper, we will consider a deformation to be continuous with respect to a parameter, if the surface converges uniformly component-wise with respect to the parameter on compact subdomains.
To consider this, we will show that the deformation exists on the level of Weierstrass data for these surfaces.

Note that with the Weierstrass data given in Theorem~\ref{thm:wdata}, we have a deformation connecting Bonnet-type surfaces to the Enneper-type surface while keeping the planar curvature line condition, as
    \[
        \lim_{\alpha \searrow 0} h_{(\alpha ,\beta)} =  \lim_{\alpha \searrow 0} \frac{\alpha}{\beta} \coth \frac{\alpha z}{2} = \frac{2}{\beta z} = h_{(0,\beta)}.
    \]
    \[
        \lim_{\alpha \searrow 0} \eta_{(\alpha ,\beta)}
            = \lim_{\alpha \searrow 0} \frac{2\beta}{\alpha^2} \sinh^2 \frac{\alpha z}{2} \dif{z}
            = \frac{\beta z^2}{2} \dif{z}
            = \eta_{(0,\beta)}.
    \]
Thus, we only need to show that there exists a continuous deformation keeping the planarity of curvature lines to the trivial Enneper surface and the catenoid.

\subsection{Deformation to the trivial Enneper surface}
To obtain a deformation to the trivial Enneper surface, let us reconsider the case $\alpha>0$ and $\beta>0$, and take solutions to \eqref{eqn:ode} as
    \[\begin{cases}
        f = \frac{\beta}{\alpha} \sinh \alpha u\\
        g = -\frac{\beta}{\alpha} \sin \alpha v.
    \end{cases}\]
Taking the path $\beta = \frac{\alpha^2}{2}$ (see Figure~\ref{fig:path}), we then have that
    \[
        e^{2\omega} = \frac{1}{4}(\cosh \alpha u + \cos \alpha v)^2.
    \]
Thus we may recover the holomorphic $1$-form of the Weierstrass data to be
    \[
        \eta_\alpha = - \cosh^2 \frac{\alpha z}{2} \dif{z}
    \]
so that by \eqref{eqn:etatoh},
    \[
        h_\alpha = \frac{2}{\alpha} \tanh \frac{\alpha z}{2}.
    \]
With these Weierstrass data, we observe that
    \[
        \lim_{\alpha \searrow 0} h_\alpha = z = h_0 \quad\text{and}\quad \lim_{\alpha \searrow 0} \eta_\alpha = - \dif{z} = \eta_0,
    \]
corresponding to the Weierstrass data of the trivial Enneper surface.
\begin{figure}
    \includegraphics{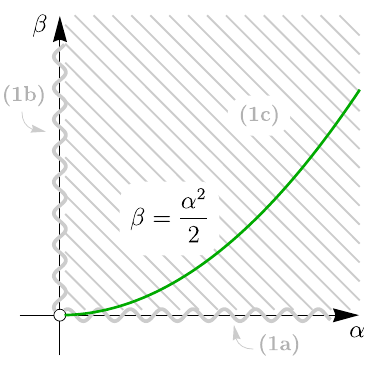}
    \caption{Path on classification diagram taken to find a deformation to the trivial Enneper surface.}
    \label{fig:path}
\end{figure}

In fact, one can obtain explicit parametrization $X_\alpha$ using the Weierstrass data $(h_\alpha, \eta_\alpha)$ and choosing the correct initial condition as
    \[
        X_\alpha = -\tfrac{1}{2\alpha^2}\begin{pmatrix}
            2(\cosh{\alpha u} \cos{\alpha v} - 1)\\
            \alpha(\sinh{\alpha u} \cos{\alpha v} + \alpha u)\\
            \alpha(\cosh{\alpha u} \sin{\alpha v} + \alpha v)
        \end{pmatrix},
    \]
and check that
    \[
        \lim_{\alpha \searrow 0} X_\alpha =
            -\begin{pmatrix}
                \tfrac{1}{2}(u^2 - v^2)\\
                u\\
                v
            \end{pmatrix}.
    \]
\subsection{Deformation to the catenoid}
The Weierstrass data given Theorem~\ref{thm:wdata} alone are not sufficient to obtain a deformation to the catenoid.
This is due to the fact that we have normalized the initial conditions of $f$ and $g$ to have zeroes, which excludes the case of catenoids as the solutions to \eqref{eqn:ode} corresponding to the case of catenoids, i.e.\ $f = e^{-\alpha u}$, have no zeroes.

To find a suitable alternate initial condition, we observe that in every case of Theorem~\ref{thm:analytic}, $g$ has a zero, while there is always some $u_0$ such that $f(u_0) = 1$.
Therefore, in considering the problem of continuous deformation, we will revisit the problem and solve for $f$ using \eqref{eqn:ode1} and \eqref{eqn:ode3} assuming that $f(0) = 1$.

However note that in Theorem~\ref{thm:analytic}, $f$ corresponding to case (1a) already satisfies the new initial condition, while $f$ in case (1b) can be easily modified to be $f = -\beta u + 1$ with $g = -\beta v$.
So let us consider case (1c), where $\alpha > 0$ and $\beta > 0$.
Then using the general solution \eqref{eqn:gen1}, we have 
    \[
        1 = f(0) = C_1 + C_2
    \]
so that \eqref{eqn:cicond} implies
    \[
        C_1 = \frac{\alpha - \sqrt{\alpha^2 + \beta^2}}{2\alpha}.
    \]
Therefore, we have
    \[
        f = \cosh \alpha u - \frac{\sqrt{\alpha^2 + \beta^2}}{\alpha} \sinh \alpha u.
    \]
Summarizing:
\begin{lemma}
    Explicit solutions to \eqref{eqn:ode} with $f(0) = 1$ and $g(0) = 0$ are given by
    \begin{description}
        \item[Case (1a)] For $\alpha > 0$ and $\beta = 0$, the solutions are given by
            \[\begin{cases}
                f_{(\alpha,0)} = e^{-\alpha u}\\
                g_{(\alpha,0)} = 0.
                \end{cases}\]
        \item[Case (1b)] For $\alpha = 0$ and $\beta > 0$, we have
            \[\begin{cases}
                f_{(0,\beta)} = -\beta u + 1\\
                g_{(0,\beta)} = -\beta v.
            \end{cases}\]
        \item[Case (1c)] For $\alpha > 0$ and $\beta > 0$, we calculate that
            \[\begin{cases}
                f_{(\alpha,\beta)} = \cosh \alpha u - \frac{\sqrt{\alpha^2 + \beta^2}}{\alpha} \sinh \alpha u \\
                g_{(\alpha,\beta)} = -\frac{\beta}{\alpha} \sin \alpha v.
            \end{cases}\]
    \end{description}
    With these solutions, we have
        \begin{gather*}
            \lim_{\alpha \to 0} f_{(\alpha,\beta)} = f_{(0,\beta)}, \quad \lim_{\alpha \to 0} g_{(\alpha,\beta)} = g_{(0,\beta)}\\
            \lim_{\beta \to 0} f_{(\alpha,\beta)} = f_{(\alpha,0)}, \quad \lim_{\beta \to 0} g_{(\alpha,\beta)} = g_{(\alpha,0)}.
        \end{gather*}
\end{lemma}

To obtain Weierstrass data for the case $\alpha >0$ and $\beta >0$, note that
    \[
        |\tilde{\eta}_{(\alpha,\beta)}|^2 = e^{2\omega} = \left(\tfrac{1}{\alpha^2}( \alpha \sinh \alpha u - \sqrt{\alpha^2 + \beta^2} \cosh \alpha u + \beta\cos \alpha v)\right)^2.
    \]
Using Fact~\ref{fact:recover} and Remark~\ref{remark:recover}, we take
    \[
        \tilde{\eta}_{(\alpha,\beta)} = \tfrac{\alpha - i \beta}{\alpha^2(\alpha + i \beta)}(\alpha \sinh \alpha z - \sqrt{\alpha^2 + \beta^2} \cosh \alpha z + \beta) \dif{z}.
    \]
Thus, taking the correct constant of integration, we have
    \[
        \tilde{h}_{(\alpha,\beta)} = \frac{\alpha(\alpha + i \beta)(\beta e^{\alpha z} + \sqrt{\alpha^2 + \beta^2} - \alpha)}{\beta (\alpha - i \beta) (\alpha - \beta \sinh \alpha z)}.
    \]
Let us define constants $r \in (0, \infty)$ and $\theta \in (0, \tfrac{\pi}{2})$ so that $\alpha = r \cos \theta$ and $\beta = r \sin \theta$.
Then we have
    \[
        \tilde{h}_{(r,\theta)} = \frac{2e^{2i \theta}\cos{\theta}}{e^{-r \cos{\theta}\, z}(\cos{\theta}+1) - \sin{\theta}}
    \]
and
    \[
        \tilde{\eta}_{(r,\theta)} = \frac{e^{-2i\theta}(\cos\theta \sinh{(r \cos \theta \, z)} - \cosh{(r \cos \theta \, z)} + \sin \theta)}{r \cos^2 \theta} \dif{z}.
    \]

For these Weierstrass data, we can verify that
    \[
        \lim_{\theta \searrow 0}  \tilde{h}_{(r,\theta)} = e^{r z} \quad\text{and}\quad
        \lim_{\theta \searrow 0}  \tilde{\eta}_{(r,\theta)} = -\frac{1}{r}e^{-r z}
    \]
while
    \[
        \lim_{\theta \nearrow \frac{\pi}{2}}  \tilde{h}_{(r,\theta)} = \frac{2}{rz - 1}
        \quad\text{and}\quad
        \lim_{\theta \nearrow \frac{\pi}{2}}  \tilde{\eta}_{(r,\theta)} = \frac{(r z - 1)^2}{2r}.
    \]
Therefore, the Weierstrass data $(\tilde{h}_{(r,\theta)},  \tilde{\eta}_{(r,\theta)} \dif{z})$ provides a continuous deformation connecting catenoid, Bonnet-type surfaces, and Enneper-type surfaces, allowing us to conclude as in the next theorem (see also Figure~\ref{fig:deformation}).
\begin{theorem}\label{thm:deformation}
    There exists a continuous deformation consisting exactly of the zero mean curvature surfaces with planar curvature lines in isotropic $3$-space.
\end{theorem}

\begin{figure}
    \includegraphics[width=0.8\linewidth]{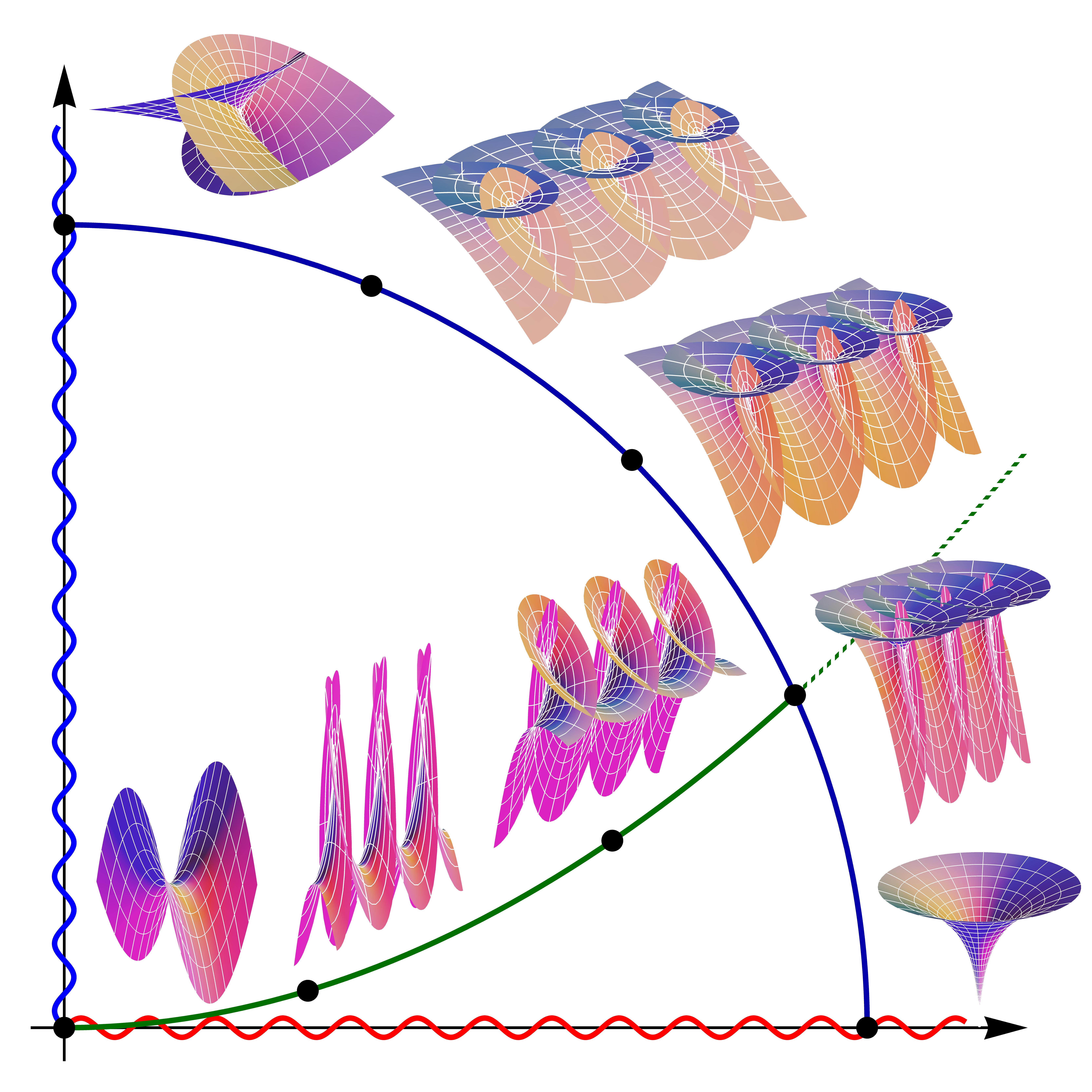}
    \caption{Continuous deformation of zero mean curvature surfaces keeping planar curvature lines}
    \label{fig:deformation}
\end{figure}

\section{Zero mean curvature surfaces that are also affine minimal}\label{sect:six}
Zero mean curvature surfaces in Euclidean space and Minkowski space with planar curvature lines enjoy a certain relationship to those zero mean curvature surfaces that are also affine minimal \cite{thomsen_uber_1923, manhart_bonnet-thomsen_2015, akamine_analysis_2020}.
In this section, investigate their relationship to zero mean curvature surfaces with planar curvature lines and recover the result by Strubecker \cite[Satz~9]{strubecker_uber_1977}.
Using this result. we will obtain a classification of all zero mean curvature surfaces in isotropic $3$-space that are also affine minimal, and conclude that they also constitute a $1$-parameter family of surfaces.

Let $X : \Sigma \to \mathbb{I}^3$ be a non-planar zero mean curvature immersion in isotropic $3$-space viewed as an affine surface.
If we define
    \begin{align*}
        \tilde{L} &:= \det(X_u, X_v, X_{uu})\\
        \tilde{M} &:= \det(X_u, X_v, X_{uv})\\
        \tilde{N} &:= \det(X_u, X_v, X_{vv}),
    \end{align*}
then for non-degenerate surfaces such that $\tilde{L}\tilde{N}-\tilde{M}^2$ does not vanish on $\Sigma$, the Berwald--Blaschke metric of the affine surface is given by
    \[
        \dif{\tilde{s}}^2 = \frac{1}{|\tilde{L}\tilde{N}-\tilde{M}^2|^{1/4}}(\tilde{L} \dif{u}^2 + 2\tilde{M}\dif{u}\dif{v} + \tilde{N} \dif{v}^2).
    \]
In particular, since zero mean curvature implies that the Gaussian curvature is negative, we may assume without loss of generality that $(u,v) \in \Sigma$ are asymptotic coordinates, i.e., we normalize the Hopf differential so that $Q = -\frac{i}{2}$.
Then the Gauss-Weingarten equations read
    \begin{equation}\label{eqn:GW2}
        \begin{cases}
        X_{uu} = \omega_u X_u - \omega_v X_v\\
        X_{uv} = \omega_v X_u + \omega_u X_v + \mathfrak{p}\\
        X_{vv} = -\omega_u X_u + \omega_v X_v\\
        n_u = - e^{-2\omega} X_v\\
        n_v = - e^{-2\omega} X_u
        \end{cases}
    \end{equation}
with the compatibility condition
    \[
        \omega_{uu} + \omega_{vv} = 0.
    \]

Under these assumptions, we have $\tilde{L} = \tilde{N} = 0$.
To find $\tilde{M}$, we first split $\mathbb{I}^3 = \mathbb{R}^2 \oplus \langle \mathfrak{p} \rangle$, and write $X = x + \chi \mathfrak{p}$ for some $x : \Sigma \to \mathbb{R}^2$ and $\chi : \Sigma \to \mathbb{R}$.
Then, we calculate
    \[
        \tilde{M} = \det(X_u, X_v, \omega_v X_u + \omega_u X_v + \mathfrak{p}) = \det(X_u, X_v, \mathfrak{p}) = \det(x_u, x_v, \mathfrak{p}),
    \]
and thus represents the area of the parallelogram spanned by $x_u$ and $x_v$ in $\mathbb{R}^2$.
By conformality \eqref{eqn:conformal}, $x_u$ and $x_v$ span a square with side length $e^{\omega}$, and thus the metric becomes
    \[
        \dif{\tilde{s}}^2 = 2\sqrt{\tilde{M}}\dif{u}\dif{v} = 2e^\omega\dif{u}\dif{v}.
    \]
The affine normal vector field $\tilde{n}$ is then given by
    \[
        \tilde{n} := e^{-\omega} X_{uv},
    \]
from which the affine shape operator $\tilde{S}$ of $X$ follows via
    \[
        \dif{\tilde{n}} = -\dif{X} \circ \tilde{S}.
    \]
Then the definitions of affine mean curvature $\tilde{H}$ and affine Gaussian curvature $\tilde{K}$ can be given as
    \[
        \tilde{H} = \frac{1}{2}\tr{\tilde{S}}, \quad\text{and}\quad \tilde{K} = \det \tilde{S}.
    \]

\begin{figure}
    \begin{minipage}{0.49\linewidth}
    	\centering
	\includegraphics[width=0.8\textwidth]{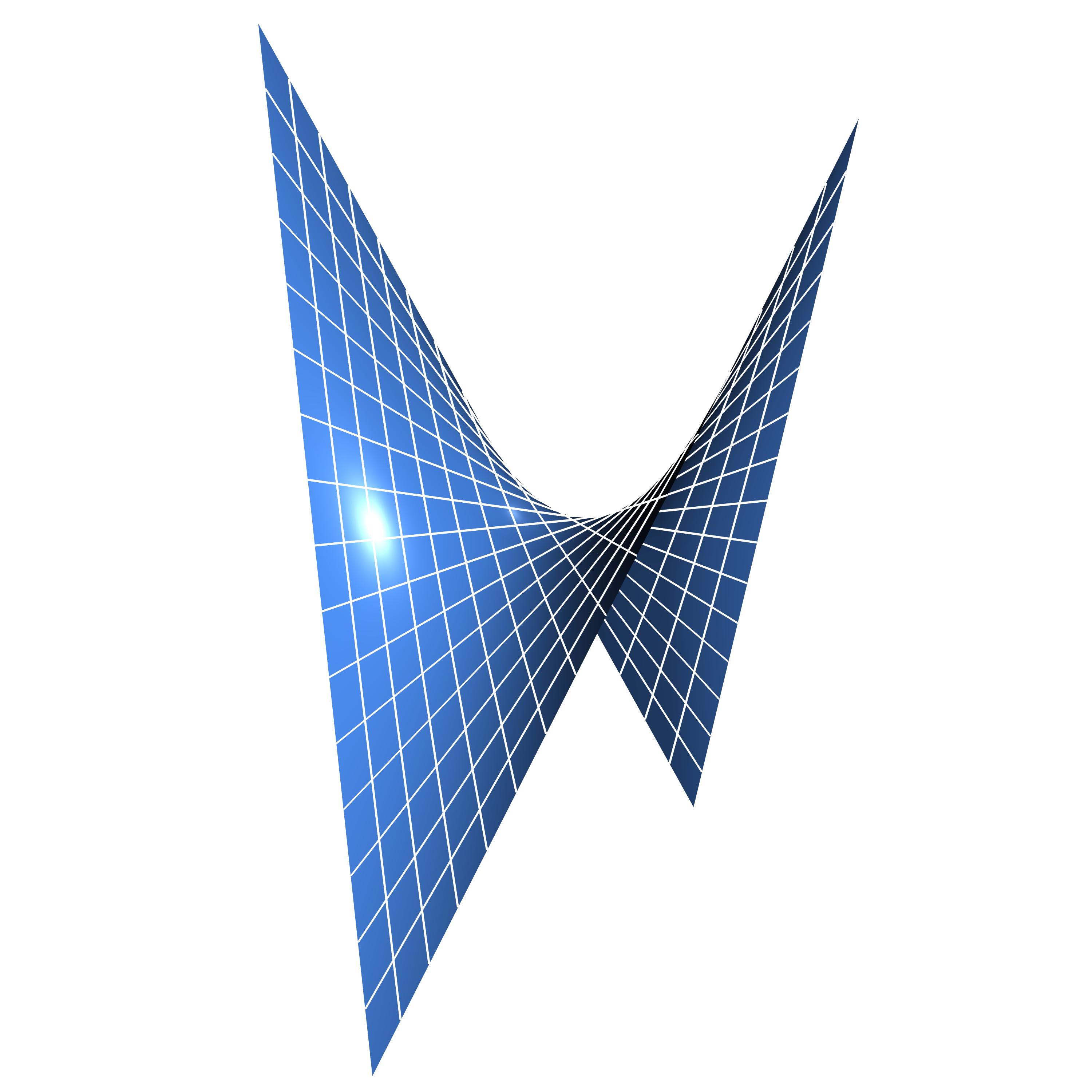}
    \end{minipage}
    \begin{minipage}{0.49\linewidth}
    	\centering
	\includegraphics[width=0.6\textwidth]{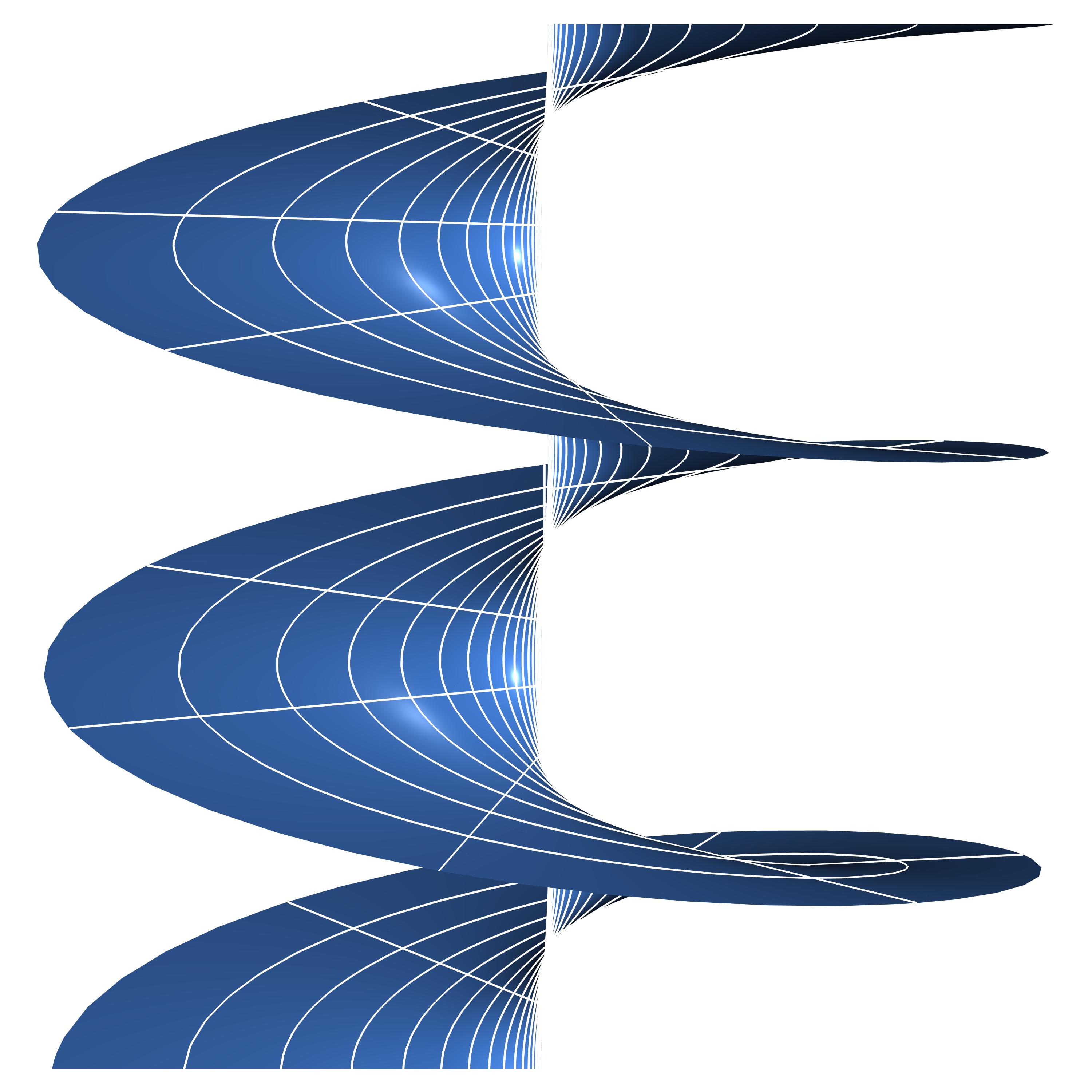}
    \end{minipage}
    
    \vspace{5pt}
    \begin{minipage}{0.49\linewidth}
    	\centering
	\includegraphics[width=0.8\textwidth]{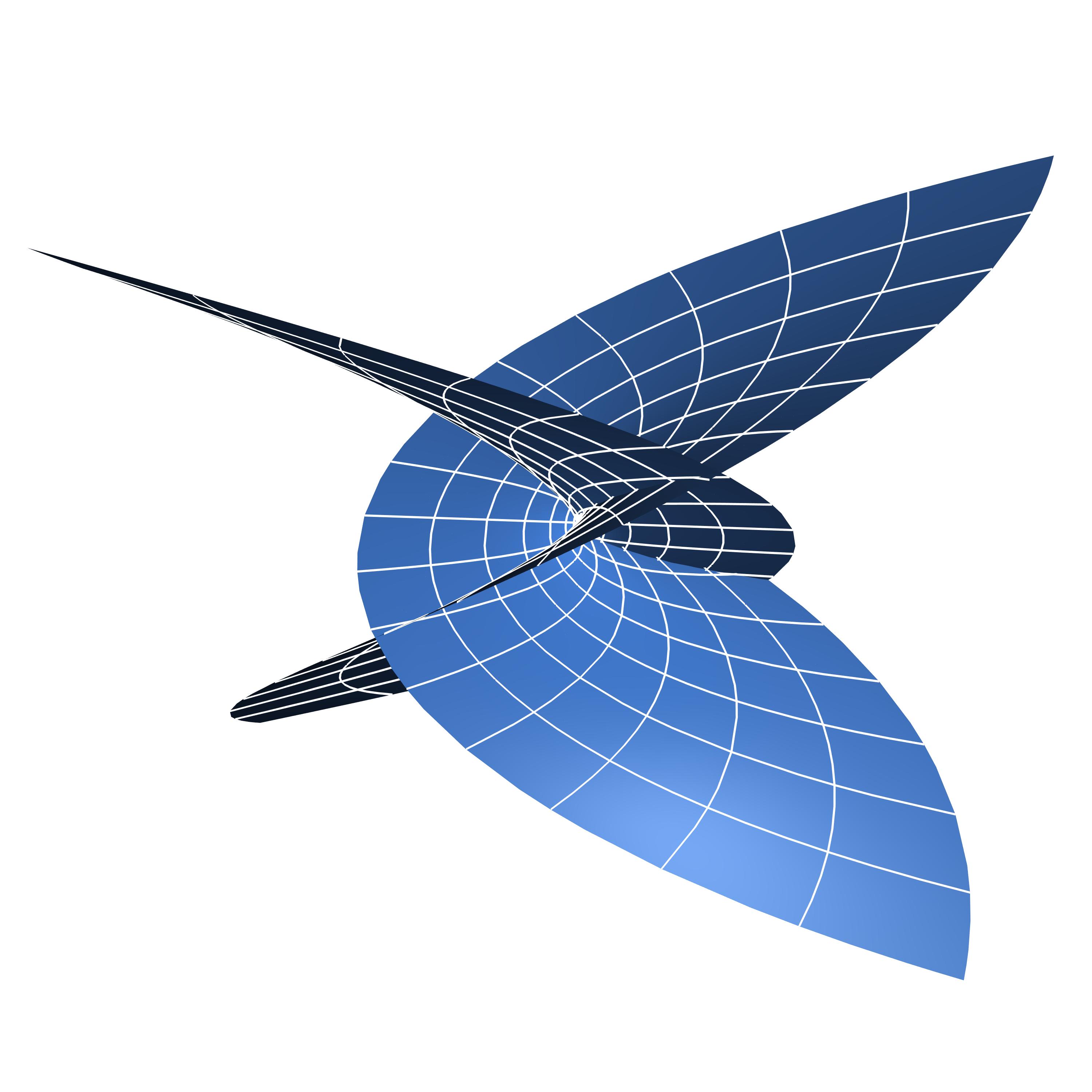}
    \end{minipage}
    \begin{minipage}{0.49\linewidth}
    	\centering
	\includegraphics[width=0.8\textwidth]{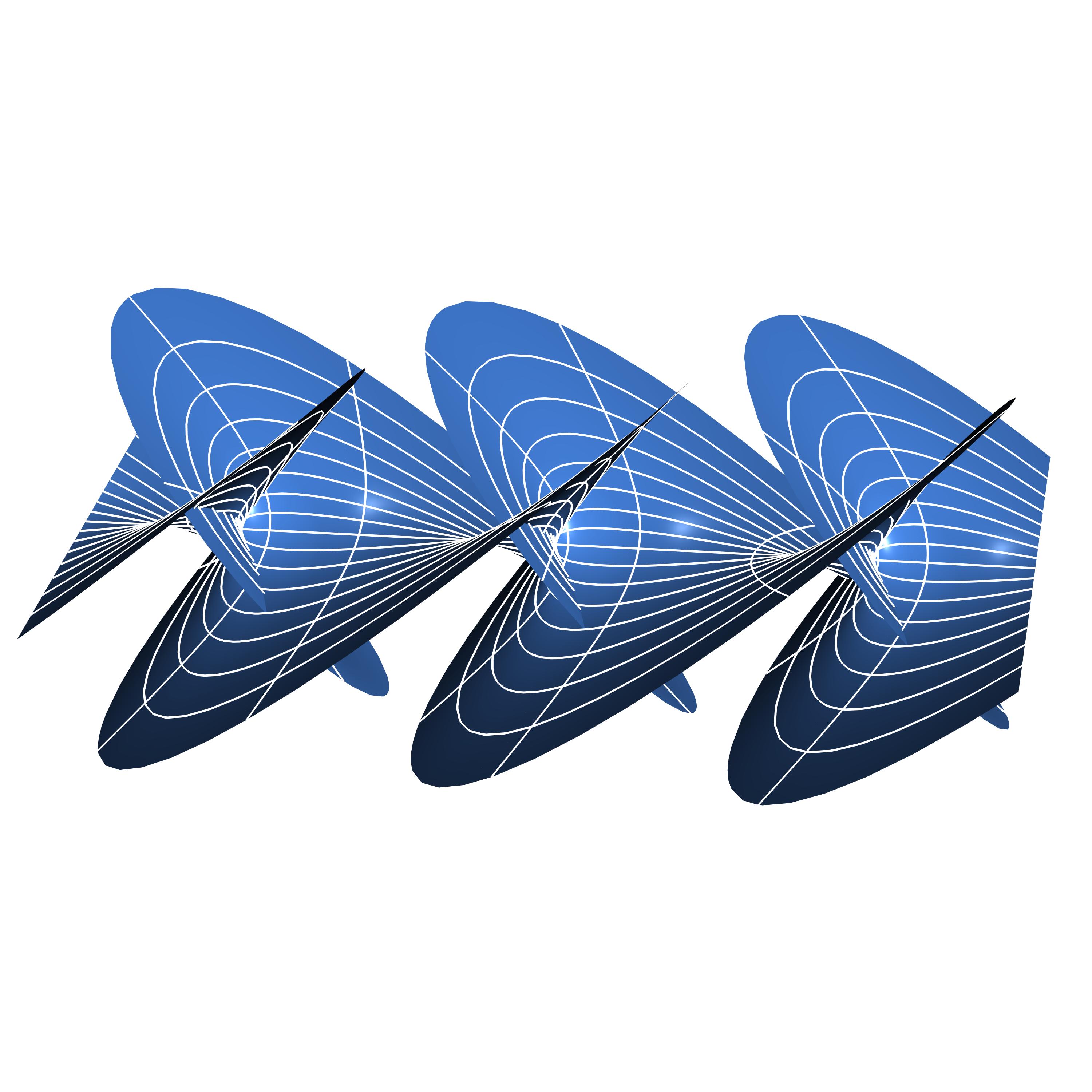}
    \end{minipage}

    \caption{Zero mean curvature surfaces that are also affine minimal in isotropic space: trivial Enneper surface and helicoid on the top row; Enneper surface and Thomsen-type surface on the bottom row.}
    \label{fig:examples2}
\end{figure}
In particular, we can verify directly that
    \begin{align*}
        \tilde{n}_u &= e^{-\omega} \big((\omega_{uv} + \omega_u \omega_v) X_u + (\omega_{uu} - \omega_v^2 ) X_v\big) \\
        \tilde{n}_v &= e^{-\omega}\big( (\omega_{vv} - \omega_u^2) X_u + (\omega_{uv} + \omega_u\omega_v) X_v \big),
    \end{align*}
and thus
    \[
        \tilde{H} = -e^{-\omega}(\omega_{uv} + \omega_u \omega_v).
    \]
Therefore, every zero mean curvature surface in isotropic space that is also affine minimal corresponds to the solution to the following system of partial differential equations:
    \begin{subnumcases}{\label{eqn:affpde}}
        \omega_{uu} + \omega_{vv} = 0, &\text{(compatibility condition)} \label{eqn:affpde1}\\
        \omega_{uv} + \omega_u \omega_v = 0, &\text{(zero mean curvature and affine minimality condition)} \label{eqn:affpde2}
    \end{subnumcases}
which is identical to the system of partial differential equations that corresponds to zero mean curvature surfaces with planar curvature lines \eqref{eqn:pde}.

\begin{figure}
    \includegraphics[width=0.8\linewidth]{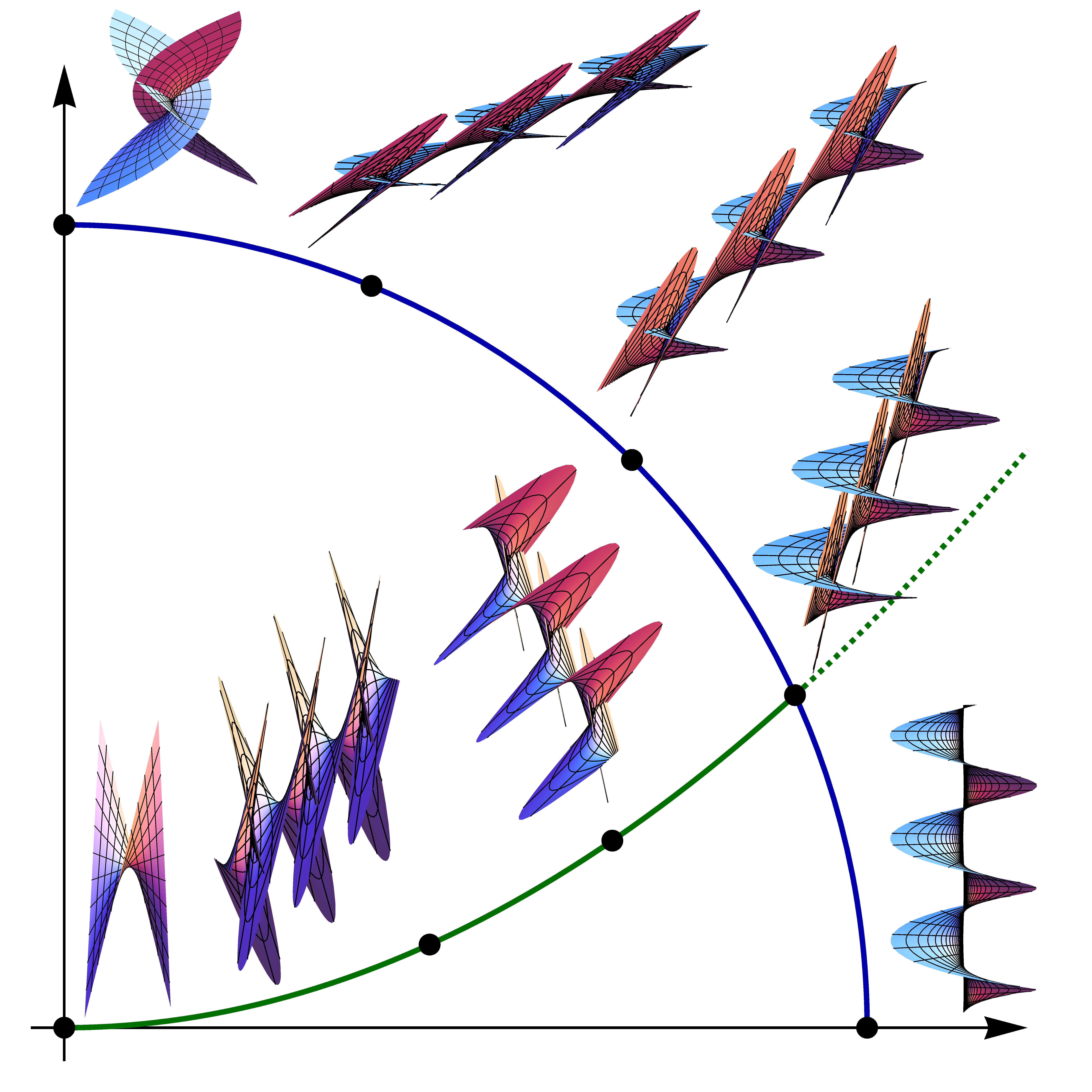}
    \caption{Continuous deformation of zero mean curvature surfaces keeping affine minimality}
    \label{fig:deformation2}
\end{figure}
Thus, every zero mean curvature surface with planar curvature lines (with normalized Hopf differerntial factor $Q = -\frac{1}{2}$) corresponds to a zero mean curvature surface that is also affine minimal (with normalized Hopf differential factor $Q = -\frac{i}{2}$), allowing us to conclude as follows from Remark~\ref{remark:conjugate}:
\begin{theorem}[{\cite[Satz~9]{strubecker_uber_1977}}]\label{thm:conjugate}
    Every zero mean curvature surface in isotropic $3$-space that is also an affine minimal surface is a conjugate zero mean curvature surface of a zero mean curvature surface with planar curvature lines.
\end{theorem}
Using the classification of zero mean curvature surfaces with planar curvature lines in Theorem~\ref{thm:wdata}, we also obtain a classficiation of all zero mean curvature surfaces in isotropic $3$-space that are also affine minimal (see also Figure~\ref{fig:examples2}):
	\begin{theorem}\label{thm:wdata2}
	    Let $X : \Sigma \to \mathbb{I}^3$ be a zero mean curvature immersion that is also affine minimal.
	    Then $X$ must be a piece of one, and only one, of
	        \begin{itemize}
	            \item plane $(0, 1 \dif{z})$,
	            \item trivial Enneper-type surface $(z, i\dif{z})$,
	            \item helicoid $(e^{z}, ie^{-z} \dif{z})$,
	            \item Enneper-type surface $(\frac{1}{z}, i z^2 \dif{z})$, or
	            \item one of Thomsen-type surfaces $(\frac{\alpha}{\beta} \coth \frac{\alpha z}{2}, i \frac{2\beta}{\alpha^2} \sinh^2 \frac{\alpha z}{2} \dif{z})$,
	        \end{itemize}
	    given with the respective Weierstrass data, up to isometries and homotheties of the isotropic space.
	\end{theorem}

Furthermore, we can also conclude from Theorem~\ref{thm:deformation} as follows, obtaining the isotropic space analogue of results in \cite{barthel_thomsensche_1980, cho_deformation_2018, akamine_analysis_2020, schaal_ennepersche_1973}  (see also Figure~\ref{fig:deformation2}):
\begin{theorem}\label{thm:deformation2}
    There exists a continuous deformation consisting exactly of the zero mean curvature surfaces that are also affine minimal in isotropic $3$-space.
\end{theorem}


\end{document}